\documentclass{amsart}

\usepackage{times}
\usepackage{amsmath, amsthm, amssymb, amsfonts}
\usepackage{mathrsfs}
\usepackage{latexsym}
\usepackage[all]{xy}
\SelectTips {cm}{}
\usepackage{graphicx}
\usepackage{comment}
\usepackage{appendix}
\usepackage{amscd}
\usepackage{a4wide}
\usepackage{mathrsfs}
\usepackage{pdfsync}
\usepackage{pstricks}
\usepackage{pstricks-add}
\usepackage{stmaryrd}

\usepackage{color}
\definecolor{Chocolat}{rgb}{0.36, 0.2, 0.09}
\definecolor{BleuTresFonce}{rgb}{0.215, 0.215, 0.36}

\usepackage[colorlinks,final,backref=page,hyperindex]{hyperref}
\hypersetup{citecolor=BleuTresFonce, linkcolor=Chocolat}

\theoremstyle{plain}

\newtheorem{theo}{Theorem}[section]
\newtheorem{prop}[theo]{Proposition}
\newtheorem{lemm}[theo]{Lemma}
\newtheorem{coro}[theo]{Corollary}
\newtheorem*{theointro}{Theorem}

\theoremstyle{definition}

\newtheorem{defi}{Definition}[section]

\theoremstyle{remark}

\newtheorem*{examples}{Examples}

\newtheorem*{remark}{Remark}

\long\def\forget#1\relax  

\def\Ai{A_{\infty}}

\def\dd{\delta}
\def\DD{\Delta}

\def\aa{\alpha}
\def\bb{\beta}

\def\t{\otimes}

\def\gr{\mathrm{gr}}

\def\B{\mathrm{B}}

\def\NN{{\mathbb{N}}}
\def\ZZ{{\mathbb{Z}}}

\def\CCC{{\mathcal{C}}}

\def\PP{{\mathcal{P}}}

\def\TTT{{\mathcal{T}}}

\def\Id{\mathrm{Id }}
\def\Hom{\mathrm{Hom}}
\def\End{\mathrm{End}}

\def\Tw{\mathrm{Tw}}

\def\epi{\twoheadrightarrow}
\def\mono{\rightarrowtail}

\def\Im{\mathop{\rm Im }}
\def\Ker{\mathop{\rm Ker }}

\def\Hom{\mathrm{\rm Hom }}

\def\End{\mathop{\rm End }}

\def\id{\mathrm{ id }}
\def\I{\mathrm{ I }}

\def\Ai{A_{\infty}}

\def\ac{^{\scriptstyle \textrm{!`}}}

\def\KK{\mathbb{K}}

\def\Sy{{\mathbb S}}

\def\arbreA{\vcenter{\xymatrix@R=3pt@C=3pt{
&& \\
&*{}\ar@{-}[ur] \ar@{-}[ul] \ar@{-}[d]     &\\
&&
}}}

\def\arbreA{\vcenter{\xymatrix@R=3pt@C=3pt{
&& \\
&*{}\ar@{-}[ur] \ar@{-}[ul] \ar@{-}[d]     &\\
&&
}}}

\def\arbreAgrand{\vcenter{\xymatrix@R=30pt@C=30pt{
&& \\
&*{}\ar@{-}[ur] \ar@{-}[ul] \ar@{-}[d]     &\\
&&
}}}

\def\arbreBA{\vcenter{\xymatrix@R=2pt@C=2pt{
&&&&\\
&&&*{}\ar@{-}[ul] & \\
&&*{}\ar@{-}[uurr] \ar@{-}[uull] \ar@{-}[d]     &&\\
&&&&
}}}

\def\arbreAB{\vcenter{\xymatrix@R=2pt@C=2pt{
&&&&\\
&*{}\ar@{-}[ur] &&& \\
&&*{}\ar@{-}[uurr] \ar@{-}[uull] \ar@{-}[d]     &&\\
&&&&
}}}

\def\arbreBB{\vcenter{\xymatrix@R=2pt@C=2pt{
&&*{}&&\\
&&&& \\
&&*{}\ar@{-}[uurr] \ar@{-}[uull] \ar@{-}[d] \ar@{-}[uu]     &&\\
&&&&
}}}

\def\arbreBBbis{\vcenter{\xymatrix@R=4pt@C=4pt@M=0pt{
&&*{}&&\\
&&&& \\
&&*{}\ar@{-}[uurr] \ar@{-}[uull] \ar@{-}[d] \ar@{-}[uu]     &&\\
&&&& }}}

\def\arbreABC{\vcenter{\xymatrix@R=1pt@C=1pt{
&&&&&&\\
&*{}\ar@{-}[ur] &&&&& \\
&&*{}\ar@{-}[uurr] &&&&\\
&&&*{}\ar@{-}[uuurrr] \ar@{-}[uuulll] \ar@{-}[d] &&&\\
&&&&&&
}}}

\def\arbreBAC{\vcenter{\xymatrix@R=1pt@C=1pt{
&&&&&&\\
&&&*{}\ar@{-}[ul] &&& \\
&&*{}\ar@{-}[uurr] &&&&\\
&&&*{}\ar@{-}[uuurrr] \ar@{-}[uuulll] \ar@{-}[d] &&&\\
&&&&&&
}}}

\def\arbreACA{\vcenter{\xymatrix@R=1pt@C=1pt{
&&&&&&\\
&*{}\ar@{-}[ur] &&&&*{}\ar@{-}[ul] & \\
&&&&&&\\
&&&*{}\ar@{-}[uuurrr] \ar@{-}[uuulll] \ar@{-}[d] &&&\\
&&&&&&
}}}

\def\arbreACB{\vcenter{\xymatrix@R=1pt@C=1pt{
&&&&&&\\
&*{}\ar@{-}[ur] &&&&& \\
&&&&*{}\ar@{-}[uull] &&\\
&&&*{}\ar@{-}[uuurrr] \ar@{-}[uuulll] \ar@{-}[d] &&&\\
&&&&&&
}}}

\def\arbreBCA{\vcenter{\xymatrix@R=1pt@C=1pt{
&&&&&&\\
&&&&&*{}\ar@{-}[ul] & \\
&&*{}\ar@{-}[uurr] &&&&\\
&&&*{}\ar@{-}[uuurrr] \ar@{-}[uuulll] \ar@{-}[d] &&&\\
&&&&&&
}}}

\def\arbreCAB{\vcenter{\xymatrix@R=1pt@C=1pt{
&&&&&&\\
&&&*{}\ar@{-}[ur] &&& \\
&&&&*{}\ar@{-}[uull] &&\\
&&&*{}\ar@{-}[uuurrr] \ar@{-}[uuulll] \ar@{-}[d] &&&\\
&&&&&&
}}}

\def\arbreCBA{\vcenter{\xymatrix@R=1pt@C=1pt{
&&&&&&\\
&&&&&*{}\ar@{-}[ul] & \\
&&&&*{}\ar@{-}[uull] &&\\
&&&*{}\ar@{-}[uuurrr] \ar@{-}[uuulll] \ar@{-}[d] &&&\\
&&&&&&
}}}

\def\arbreAAC{\vcenter{\xymatrix@R=1pt@C=1pt{
&&&&&&\\
&&&&&& \\
&&*{}\ar@{-}[uurr]\ar@{-}[uu]  &&&&\\
&&&*{}\ar@{-}[uuurrr] \ar@{-}[uuulll] \ar@{-}[d] &&&\\
&&&&&&
}}}

\def\arbreCAC{\vcenter{\xymatrix@R=1pt@C=1pt{
&&&&&&\\
&&&*{}\ar@{-}[ul]\ar@{-}[ur] &&& \\
&&&&&&\\
&&&*{}\ar@{-}[uuurrr] \ar@{-}[uuulll]\ar@{-}[uu]  \ar@{-}[d] &&&\\
&&&&&&
}}}

\def\arbreACC{\vcenter{\xymatrix@R=1pt@C=1pt{
&&&&&&\\
&*{}\ar@{-}[ur] &&&&& \\
&&&&&&\\
&&&*{}\ar@{-}[uuurrr] \ar@{-}[uuulll]\ar@{-}[uuu]  \ar@{-}[d] &&&\\
&&&&&&
}}}

\def\arbreCBB{\vcenter{\xymatrix@R=1pt@C=1pt{
&&&&&&\\
&&&&&& \\
&&&&*{}\ar@{-}[uull]\ar@{-}[uu]  &&\\
&&&*{}\ar@{-}[uuurrr] \ar@{-}[uuulll] \ar@{-}[d] &&&\\
&&&&&&
}}}

\def\arbreCCA{\vcenter{\xymatrix@R=1pt@C=1pt{
&&&&&&\\
&&&&&*{}\ar@{-}[ul] & \\
&&&&&&\\
&&&*{}\ar@{-}[uuurrr] \ar@{-}[uuulll]\ar@{-}[uuu]  \ar@{-}[d] &&&\\
&&&&&&
}}}

\def\arbreBBC{\vcenter{\xymatrix@R=1pt@C=1pt{
&&&&&&\\
&&&&&& \\
&&*{}\ar@{-}[uu] \ar@{-}[uurr] &&&&\\
&&&*{}\ar@{-}[uuurrr] \ar@{-}[uuulll] \ar@{-}[d] &&&\\
&&&&&&
}}}

\def\arbreCCC{\vcenter{\xymatrix@R=1pt@C=1pt{
&&&&&&\\
&&&&&& \\
&&&&&&\\
&&&*{}\ar@{-}[uuurrr]  \ar@{-}[uuulll]\ar@{-}[uuul] \ar@{-}[uuur]  \ar@{-}[d] &&&\\
&&&&&&
}}}

\def\arbreCCCbis{\vcenter{\xymatrix@R=4pt@C=4pt@M=0pt{
&&&&&&\\
&&&&&& \\
&&&&&&\\
&&&*{}\ar@{-}[uuurrr]  \ar@{-}[uuulll]\ar@{-}[uuul] \ar@{-}[uuur]  \ar@{-}[d] &&&\\
&&&&&& }}}

\def\arbreADAB{\vcenter{\xymatrix@R=5pt@C=5pt{
 &&&&&&&&\\
&*{}\ar@{-}[ur] &&&&*{}\ar@{-}[ur] && & \\
&&&&&&*{} \ar@{-}[uull] &&\\
&&&&&&&&\\
&&&&*{}\ar@{-}[uuuurrrr] \ar@{-}[uuuullll] \ar@{-}[d] &&&&\\
&&&&&&&&
}}}

\def\arbreBBDA{\vcenter{\xymatrix@R=1pt@C=1pt{
 &&&&&&\\
&&&&&*{}\ar@{-}[ul] \ar@{-}[ur]& \\
&&*{}\ar@{-}[uu] \ar@{-}[uurr] \ar@{-}[uull] &&&&\\
&&&&&&\\
&&&*{}\ar@{-}[uuurr] \ar@{-}[uul] \ar@{-}[d] &&&\\
&&&&&&
}}}

\def\arbreBBDAa{\vcenter{\xymatrix@R=1pt@C=1pt{
 \ar@{.}[rrrrrr]&&&&&&\\
&&&&& *{} \ar@{-}[ul] \ar@{-}[ur] & \\
&&*{}\ar@{-}[uu] \ar@{-}[uurr] \ar@{-}[uull] &&&&\\
&&&&&&\\
&&&*{}\ar@{-}[uuurr] \ar@{-}[uul] \ar@{-}[d] &&&\\
&&&&&&
}}}

\def\arbreBBDAb{\vcenter{\xymatrix@R=1pt@C=1pt{
 \ar@{.}[rrrr]&&&&&&\\
&&&&&*{}\ar@{-}[ul] \ar@{-}[ur]& \\
&&*{}\ar@{-}[uu] \ar@{-}[uurr]\ar@{-}[uull]   &\ar@{.}[rrr]&&&\\
&&&&&&\\
&&&*{}\ar@{-}[uuurr] \ar@{-}[uul] \ar@{-}[d] &&&\\
&&&&&&
}}}

\def\arbreBBDAc{\vcenter{\xymatrix@R=1pt@C=1pt{
 &&&& \ar@{.}[rr]&&\\
&&&&&*{}\ar@{-}[ul] \ar@{-}[ur]& \\
&&*{}\ar@{-}[uu] \ar@{-}[uurr]\ar@{-}[uull]  &&&&\\
 \ar@{.}[rrr]&&&&&&\\
&&&*{}\ar@{-}[uuurr] \ar@{-}[uul] \ar@{-}[d] &&&\\
&&&&&&
}}}

\def\arbreBBDAd{\vcenter{\xymatrix@R=1pt@C=1pt{
 &&&&&&\\
&&&&&*{}\ar@{-}[ul] \ar@{-}[ur] & \\
&&*{}\ar@{-}[uu] \ar@{-}[uurr]\ar@{-}[uull]  &&&&\\
 \ar@{.}[rrrrrr]&&&&&&\\
&&&*{}\ar@{-}[uuurr] \ar@{-}[uul] \ar@{-}[d] &&&\\
&&&&&&
}}}

\def\arbreBBDAe{\vcenter{\xymatrix@R=1pt@C=1pt{
 &&&&&&\\
&&&&&*{}\ar@{-}[ul] \ar@{-}[ur]& \\
&&*{}\ar@{-}[uu] \ar@{-}[uurr]\ar@{-}[uull]  &&&&\\
&&&&&&\\
&&&*{}\ar@{-}[uuurr] \ar@{-}[uul] \ar@{-}[d] &&&\\
 \ar@{.}[rrrrrr]&&&&&&
}}}

\def\elecun{\vcenter{\xymatrix@R=6pt@C=6pt{
&&*{}\ar@{-}[d] &&&*{}\ar@{-}[dddd] \\
&&*{}\ar@{-}[dl] \ar@{-}[dr] &&&\\
&*{}\ar@{-}[dddd] &&*{}\ar@{-}[d] &&\\
&&&\bb \ar@{-}[d]&&\\
&&&*{}\ar@{-}[dr] &&*{}\ar@{-}[dl] \\
&&&&*{}\ar@{-}[dddd] &\\
&*{}\ar@{-}[dr] \ar@{-}[dl] &&&&\\
*{}\ar@{-}[dddd] &&*{}\ar@{-}[d] &&&\\
&&\aa \ar@{-}[d]&&&\\
&&*{}\ar@{-}[dr] &&*{}\ar@{-}[dl] &\\
&&&*{}\ar@{-}[d] &&\\
&&&*{}&&
}}}

\def\elecdeux{\vcenter{\xymatrix@R=6pt@C=6pt{
&&*{}\ar@{-}[dd] &&&*{}\ar@{-}[dddddd] \\
&&&&&\\
&&*{}\ar@{-}[dl] \ar@{-}[dr] &&&\\
&*{}\ar@{-}[dr] \ar@{-}[dl] &&*{}\ar@{-}[dd] &&\\
*{}\ar@{-}[ddddddd] &&*{}\ar@{-}[d]&&&\\
&&\aa \ar@{-}[dd]&\bb \ar@{-}[d]&&\\
&&&*{}\ar@{-}[dr] &&*{}\ar@{-}[dl] \\
&&*{}\ar@{-}[dr] &&*{}\ar@{-}[dl] &\\
&&&*{}\ar@{-}[ddd] &\\
&&&&&\\
&& &&&\\
&&&*{}&&
}}}

\def\electrois{\vcenter{\xymatrix@R=6pt@C=6pt{
&&*{}\ar@{-}[dd] &&&&*{}\ar@{-}[dddddd] \\
&&&&&&\\
&&*{}\ar@{-}[dl] \ar@{-}[ddrr] &&&&\\
&*{}\ar@{-}[dl] &&*{}\ar@{-}[dl] & &&\\
*{}\ar@{-}[ddddddd] &&*{}\ar@{-}[d]&&*{}\ar@{-}[d]&\\
&&\aa \ar@{-}[d]&&\bb \ar@{-}[d]&&\\
&&*{}\ar@{-}[dr] &&*{}\ar@{-}[dl] &&*{}\ar@{-}[dl] \\
&&&*{}\ar@{-}[dr] &&*{}\ar@{-}[dl] &\\
&&&&*{}\ar@{-}[ddd] &\\
&&&&&\\
&& &&&\\
&&&*{}&&
}}}

\def\elecquatre{\vcenter{\xymatrix@R=6pt@C=6pt{
&*{}\ar@{-}[d] &&&*{}\ar@{-}[ddddd] &&*{}\ar@{-}[ddddddd] &&*{}\ar@{-}[ddddddd]  \\
&*{}\ar@{-}[dl]\ar@{-}[dr]  &&&  &&&&\\
*{}\ar@{-}[ddddd] &&*{}\ar@{-}[d] && &&&&\\
&&{\bullet}&& &= & & & &\\
&&{\bullet}\ar@{-}[d] && &&&&\\
&&*{}\ar@{-}[dr] &&*{}\ar@{-}[dl]  &&&&\\
&&&*{}\ar@{-}[d] & &&&&\\
&&&*{}& &&*{}&&*{}
}}}

\def\KzeroF{\xymatrix@R=4pt@C=4pt{
\\
\\
\\
\bullet\\
}}

\def\KunF{\xymatrix@R=4pt@C=4pt{
&&\\
&&\\
&&\\
*{}\ar@{->}[rrr]&&&*{}\\
}}

\def\KdeuxF{\xymatrix@R=4pt@C=4pt{
&& &&\\
&&*{}\ar[dll]\ar[ddrrr]  &&&&\\
 *{}\ar[dd] &&&&&\\
&&&&&*{}\ar[ddlll] \\
 *{}\ar[drr] &&&&& \\
&&*{}&&& \\
&& &&&
}}

\def\KtroisF{\xymatrix@R=0pt@C=0pt{
&*{}\ar[rrrrrrrrrr] *{}\ar@{<-}[rrrrddd] *{}\ar[ldd] &&&& &&& &&&*{}\ar@{.>}[llldddddd]  *{}\ar@{<-}[rrd] && \\
&&&& &&& &&& &&&*{}\ar@{<-}[llldd] *{}\ar[llldddddd]   \\
*{}\ar@{<-}[rrrrddd]  *{}\ar[rrdddd] &&&&& &&& &&& && \\
&&&&&*{}\ar[rrrrr] *{}\ar[ldd]  &&& &&*{}\ar[lldddd] & && \\
&&&&& &&& &&& && \\
&&&&*{}\ar[rdd] & &&& &&& & &\\
&&*{}\ar@{.>}[rrrrrr] *{}\ar@{<-}[rrrddd] &&& &&&*{}\ar@{<.}[rrd]  &&& && \\
&&&&&*{}\ar[rrr] *{}\ar[dd]  &&&*{}\ar[dd]  &&*{}\ar@{<-}[lldd] & & &\\
&&&&& &&& &&& && \\
&&&&&*{}\ar[rrr]  &&&*{} &&& && \\
}}

\def\KunCube{\xymatrix@R=4pt@C=4pt{
&&&&\\
&&&&\\
&&&&\\
*{}|\ar@{-}[rrrr]&&|&&*{}|\\
}}

\def\KdeuxCube{\xymatrix@R=6pt@C=6pt{
&&*{}\ar@{-}[ddll]\ar@{-}[dddrr]  &&\\
&*{}\ar@{-}[ddr]&  &&\\
*{}\ar@{-}[dd] & & & & \\
*{}\ar@{-}[rr] & &*{}\ar@{-}[ddr]\ar@{-}[uur] &\quad & *{}\ar@{-}[dddll]\\
*{}\ar@{-}[ddrr] & & & & \\
&*{}\ar@{-}[uur]&&& \\
&&*{}&&\\
}}

\def\KdeuxDiagdeuxA{\xymatrix@R=8pt@C=8pt{
&&\quad &             \quad            &*{}\ar@{-}[dddlll]\ar@{-}[ddddddrrrrrr]&& &&& & \\
&&&       *{}               &&& &&& & \\
&&*{}&                         &&&*{} &&& & \\
&*{}\ar@{-}[dddddd]&&*{}\ar@{-}[ddll]\ar@{-}[dr]\ar@{-}[uu] &&& &&& & \\
&&*{}\ar@{-}[dr]\ar@{-}[uu]&      &*{}\ar@{-}[dl]\ar@{-}[ddddrrrr]\ar@{-}[uurr]&& &&*{}& & \\
&*{}&&    *{}\ar@{-}[dd] &&& &&& & \\
&&*{}\ar@{-}[uu]\ar@{-}[dl]\ar@{-}[dr]&                        &&& &&& &*{} \\
&*{}&&       *{}\ar@{-}[dddrrr]&&& &&& & \\
&&*{}\ar@{-}[ur]\ar@{-}[dd]&      &&&&&*{}&  & \\
&*{}&&*{}\ar@{-}[uuuuurrrrr]\ar@{-}[uull]\ar@{-}[dd]&&& &&& & \\
&&*{}& &&&*{} &&& & \\
&&&*{} &&& &&& & \\
&&& &*{}\ar@{-}[uuulll]\ar@{-}[uuuuuurrrrrr]&& &&& &
}}

\def\KdeuxDiagdeuxB{\xymatrix@R=8pt@C=8pt{
&&\quad &             \quad            &*{}\ar@{-}[dddlll]\ar@{-}[ddddddrrrrrr]&& &&& & \\
&&&       *{}               &&& &&& & \\
&&*{}&                         &&&*{} &&& & \\
&*{}\ar@{-}[dddddd]&&*{}\ar@{-}[ddll]\ar@{-}[dr]\ar@{-}[uu] &&& &&& & \\
&&*{}\ar@{-}[dr]\ar@{-}[uu]&      &*{}\ar@{-}[dl]\ar@{-}[ddddrrrr]\ar@{-}[uurr]&& &&*{}& & \\
&*{}&&    *{}\ar@{-}[dd] &&& &&& & \\
&&*{}\ar@{-}[dd]\ar@{-}[ul]\ar@{-}[ur]&                        &&& &&& &*{} \\
&*{}&&       *{}\ar@{-}[dddrrr]&&& &&& & \\
&&*{}\ar@{-}[ur]\ar@{-}[dd]&      &&&&&*{}&  & \\
&*{}&&*{}\ar@{-}[uuuuurrrrr]\ar@{-}[uull]\ar@{-}[dd]&&& &&& & \\
&&*{}& &&&*{} &&& & \\
&&&*{} &&& &&& & \\
&&& &*{}\ar@{-}[uuulll]\ar@{-}[uuuuuurrrrrr]&& &&& &
}}

\def\cpbdeuxdeux{\vcenter{\xymatrix@R=2pt@C=2pt{
&&\\
*{}\ar@{-}[u]\ar@{-}[dr]&&*{}\ar@{-}[u]\ar@{-}[dl]\\
&*{}\ar@{-}[d]&\\
&*{}\ar@{-}[dl]\ar@{-}[dr]&\\
*{}\ar@{-}[d]&&*{}\ar@{-}[d]\\
&&
}}}

\def\cpbA{\vcenter{\xymatrix@R=2pt@C=2pt{
*{}\ar@{-}[ddddd]&&*{}\ar@{-}[ddddd]\\
&&\\
&&\\
&&\\
&&\\
*{}&&*{}
}}}

\def\cpbB{\vcenter{\xymatrix@R=2pt@C=2pt{
*{}\ar@{-}[dd]&&*{}\ar@{-}[dd]\\
&&\\
*{}\ar@{-}[drr]& &*{}\ar@{-}[dll]\\
*{}\ar@{-}[dd]& &*{}\ar@{-}[dd]\\
&&\\
*{}&&*{}
}}}

\def\cpbC{\vcenter{\xymatrix@R=2pt@C=2pt{
&&&&\\
&*{}\ar@{-}[dl]\ar@{-}[dr]\ar@{-}[u]&&&\\
*{}\ar@{-}[ddd]&&*{}\ar@{-}[d]&&\\
&&*{}\ar@{-}[dr]&&*{}\ar@{-}[dl]*{}\ar@{-}[uuu]\\
&&&*{}\ar@{-}[d]&\\
&&&&
}}}

\def\cpbD{\vcenter{\xymatrix@R=2pt@C=2pt{
&&&\\
&*{}\ar@{-}[dl]\ar@{-}[dr]\ar@{-}[u]&&*{}\ar@{-}[d]*{}\ar@{-}[u]\\
*{}\ar@{-}[d]&&*{}\ar@{-}[dr]&*{}\ar@{-}[dl]\\
*{}\ar@{-}[dr]&&*{}\ar@{-}[dl]&*{}\ar@{-}[dd]\\
&*{}\ar@{-}[d]&&\\
&&&
}}}

\def\cpbE{\vcenter{\xymatrix@R=2pt@C=2pt{
&&&&\\
&&&*{}\ar@{-}[u]\ar@{-}[dl]\ar@{-}[dr]&\\
&&*{}\ar@{-}[d]&&*{}\ar@{-}[ddd]\\
*{}\ar@{-}[uuu]\ar@{-}[dr]&&*{}\ar@{-}[dl]&&\\
&*{}\ar@{-}[d]&&&\\
&&&&
}}}

\def\cpbF{\vcenter{\xymatrix@R=2pt@C=2pt{
&&&\\
&&*{}\ar@{-}[u]&\\
*{}\ar@{-}[uu]*{}\ar@{-}[dr]&*{}\ar@{-}[dl]\ar@{-}[ur]&&*{}\ar@{-}[ul]\ar@{-}[d]\\
*{}\ar@{-}[dd]&*{}\ar@{-}[dr]&&*{}\ar@{-}[dl]\\
&&*{}\ar@{-}[d]&\\
&&&
}}}

\def\cpbG{\vcenter{\xymatrix@R=2pt@C=2pt{
&&&&&\\
&*{}\ar@{-}[dl]\ar@{-}[dr]\ar@{-}[u]&&&*{}\ar@{-}[dl]\ar@{-}[dr]\ar@{-}[u]&\\
*{}\ar@{-}[d]&&*{}\ar@{-}[dr]&*{}\ar@{-}[dl]&&*{}\ar@{-}[d]\\
*{}\ar@{-}[dr]&&*{}\ar@{-}[dl]&*{}\ar@{-}[dr]&&*{}\ar@{-}[dl]\\
&*{}\ar@{-}[d]&&&*{}\ar@{-}[d]&\\
&&&&&&
}}}

\def\Kdeuxlabel{\xymatrix@R=4pt@C=4pt{
&& & \texttt{a} & && &&& \\
& &&*{}\ar[dddrrrr]\ar[ddll]&& &&& &\\
& u &\quad &  && v && & &\\
\texttt{b} &*{}\ar[dd]& && &&& && \\
{}\quad w && & Z && &&*{}\ar[dddllll]& \texttt{c} & \\
\texttt{d} &*{}\ar[ddrr]&& && &&& & \\
&y&   && &x&&  & & \\
&& &*{}&& &&& & \\
&& &\texttt{e}& &&&& & 
}}

\def\KdeuxDiag{\xymatrix@R=12pt@C=12pt{
&&& *{}\ar@{-}[ddddrrrr]\ar@{-}[dddlll]&&& & \\
&&&  &&&  &\\
&&& \texttt{a}Z && &  & \\
*{}\ar@{-}[dd]& uw &*{}\ar@{-}[uu]\ar@{-}[dr]\ar@{-}[dll]&  &&&  & \\
*{}\ar@{-}[drr]&& uy & *{}\ar@{-}[ddrr]\ar@{-}[dl]\ar@{-}[uurr] && v x &  &*{}\ar@{-}[ddddllll] \\
*{}\ar@{-}[dddrrr]& wy & *{}\ar@{-}[dd]&  &&&  & \\
&&& Z\texttt{e} &&&  & \\
&&&  &&&  & \\
&&& *{} &&&  & \\
}}

\def\Kzero{\xymatrix@R=4pt@C=4pt{
\\
\\
\\
{\bullet}\\
}}

\def\KunA{\xymatrix@R=4pt@C=4pt{
&&\\
&&\\
&&\\
*{}\ar@{-}[rr]&&*{}\\
}}

\def\PdeuxA{\xymatrix@R=4pt@C=4pt{
&&&&&\\
&&*{}\ar@{-}[dll]\ar@{-}[drr]  &&& \\
*{}\ar@{-}[dd] &&&&*{}\ar@{-}[dd] &\\
&&&&& \\
*{}\ar@{-}[drr] &&&&*{}\ar@{-}[dll] & \\
&&*{}&&& \\
}}

\def\PtroisA{\xymatrix@R=4pt@C=4pt{
&&&*{}\ar@{-}[dll]\ar@{.}[dr]\ar@{-}[r] &*{}\ar@{-}[drr]\ar@{-}[dll] &&& \\
&*{}\ar@{-}[ddl]\ar@{-}[r]&*{}\ar@{-}[dd] & &*{}\ar@{.}[ddl]\ar@{.}[drr]&&*{}\ar@{-}[dd]\ar@{-}[dr]& \\
&&& & &&*{}\ar@{.}[r]\ar@{.}[ddl]&*{}\ar@{-}[dd] \\
*{}\ar@{-}[dd]\ar@{.}[dr]&&*{}\ar@{-}[ddl]\ar@{-}[drr]&*{}\ar@{.}[dll]\ar@{.}[drr] & &&*{}\ar@{-}[dll]\ar@{-}[dr]& \\
&*{}\ar@{.}[dd]&& &*{}\ar@{-}[ddl] &*{}\ar@{.}[dd]&&*{}\ar@{-}[ddl] \\
*{}\ar@{-}[dr]\ar@{-}[r]&*{}\ar@{-}[drr]&& & &&& \\
&*{}\ar@{-}[drr]&&*{}\ar@{-}[dr] & &*{}\ar@{.}[dll]\ar@{.}[r]&*{}\ar@{-}[dll]& \\
&&&*{}\ar@{-}[r] &*{} &&& \\
}}

\def\cpbDeuxDeux#1#2{\vcenter{\xymatrix@R=2pt@C=2pt{
*{}\ar@{-}[d]&&*{}\ar@{-}[d]\\
*{}\ar@{-}[dr]&&*{}\ar@{-}[dl]\\
&#1\ar@{-}[dd]&\\
&&\\
&#2\ar@{-}[dl]\ar@{-}[dr]&\\
*{}\ar@{-}[d]&&*{}\ar@{-}[d]\\
*{}&&*{}
}}}

\def\cpbCC#1#2{\vcenter{\xymatrix@R=2pt@C=2pt{
&*{}\ar@{-}[d]&&&*{}\ar@{-}[ddd]\\
&#1\ar@{-}[dl]\ar@{-}[dr]&&&\\
*{}\ar@{-}[ddd]&&*{}\ar@{-}[d]&&\\
&&*{}\ar@{-}[dr]&&*{}\ar@{-}[dl]\\
&&&#2\ar@{-}[d]&\\
*{}&&&*{}&
}}}

\def\cpbEE#1#2{\vcenter{\xymatrix@R=2pt@C=2pt{
*{}\ar@{-}[ddd]&&&*{}\ar@{-}[d]&\\
&&&#1\ar@{-}[dl]\ar@{-}[dr]&\\
&&*{}\ar@{-}[d]&&*{}\ar@{-}[ddd]\\
*{}\ar@{-}[dr]&&*{}\ar@{-}[dl]&&\\
&#2\ar@{-}[d]&&&\\
&*{}&&&*{}
}}}

\def\cpbtroisquatre{\vcenter{\xymatrix@R=6pt@C=6pt{
&&&&\\
&*{}\ar@{-}[u]&&*{}\ar@{-}[u]&\\
&&\mu\ar@{-}[ul]\ar@{-}[uu]\ar@{-}[ur]&&\\
&&\dd\ar@{-}[u]\ar@{-}[dll]\ar@{-}[dl] \ar@{-}[drr]\ar@{-}[dr] &&\\
*{}\ar@{-}[d]&*{}\ar@{-}[d]&&*{}\ar@{-}[d]&*{}\ar@{-}[d]\\
&&&&
}}}

\def\cpbmultiple{\vcenter{\xymatrix@R=6pt@C=6pt{
&&&&    &&&&    &&&&     \\
&&&&  *{}\ar@{-}[u]  &&&&   *{}\ar@{-}[u]   &&&&    \\
*{}\ar@{-}[rrrrrrrrrrr] &*{}\ar@{-}[uu]&&*{}\ar@{-}[ur]&  *{}\ar@{-}[uu]  &*{}\ar@{-}[ul]&*{}\ar@{-}[urr]  &*{}\ar@{-}[ur]  &  *{}\ar@{-}[u]    &&*{}\ar@{-}[ull]  &*{}&    \\
&*{}&*{}&*{}&  *{} &*{}& \omega &*{}& *{}&*{}&*{}&*{}& *{}   &*{}&*{}&*{}& *{}   \\
*{}\ar@{-}[uu] \ar@{-}[rrrrrrrrrrr] &*{}&*{}&*{}&   *{} &*{}&*{}&*{}&  *{}&*{}& *{}    &*{}\ar@{-}[uu] &  *{}  \\
&&*{}\ar@{-}[d]\ar@{-}[ul]\ar@{-}[ur]  &&  *{}\ar@{-}[d] \ar@{-}[u] 
  &&*{}\ar@{-}[d] \ar@{-}[ul] \ar@{-}[u]\ar@{-}[ur]  &&&*{}\ar@{-}[d] \ar@{-}[ul] \ar@{-}[ur] &    \\
&&&&    &&&&      &&&&    \\
}}}

\def\cubesDeux{
\xymatrix@R=6pt@C=6pt{
&& && *{}\ar@{-}[ddllll]\ar@{-}[drr]\ar@{-}[dd]&& &&\\
&& && &&*{}\ar@{-}[dd] &&\\
*{}\ar@{-}[dddd]&& &&*{}\ar@{-}[ddllll]\ar@{-}[drr] &&&& \\
&& *{}\ar@{-}[ddrrrr]\ar@{-}[dd]  && &&*{}\ar@{-}[dll]&&\\
*{}&& && *{}&& && =\qquad \\
&& *{}\ar@{-}[ddrrrr]\ar@{-}[dll]  && &&*{}\ar@{-}[dd] &&\\
*{}\ar@{-}[ddrrrr]&& &&*{}&&&&\\
&& && &&*{}\ar@{-}[dll]&&  \\
&& &&*{} && &&
}}

\def\cubesTrois{
\xymatrix@R=6pt@C=6pt{
&& && *{}\ar@{-}[ddllll]\ar@{-}[drr]\ar@{-}[dd]&& \\
&& *{}\ar@{-}[dddd]&& &&*{}\ar@{-}[dd] \\
*{}\ar@{-}[dddd]&& &&*{}\ar@{-}[dll]\ar@{-}[drr] && \\
&& *{}\ar@{-}[ddrrrr]&& &&*{}\ar@{-}[dll]\\
&& && *{}\ar@{-}[dd]&& \\
&& *{}\ar@{-}[ddrrrr]\ar@{-}[dll]  && &&*{}\ar@{-}[dd] \\
*{}\ar@{-}[ddrrrr]&& &&*{}\ar@{-}[dll]&&\\
&& *{}&& &&*{}\ar@{-}[dll]  \\
&& &&*{} && 
}}

\def\cubeuwy{\xymatrix@R=8pt@C=8pt{
& u && *{}\ar@{-}[dd]\ar@{-}[dll]\ar@{-}[drr] &\ & & \\
& *{}\ar@{-}[dd]\ar@{.}[drr] && && *{}\ar@{-}[dd]\ar@{.}[dll] & \\
w & && *{}\ar@{-}[dll]\ar@{-}[drr] \ar@{.}[dd]&& & \\
 & *{}\ar@{-}[drr] && && *{}\ar@{-}[dll] & \\
 & y && *{} && & 
}}

\newcommand{\pAB}{\partial^A_B}

\newcommand{\Bk}{\B_\kappa}
\newcommand{\Ok}{\Omega_\kappa}

\newcommand{\pullback}{\mbox{\LARGE{$\lrcorner$}}}
\newcommand{\pushout}{\mbox{\LARGE{$\ulcorner$}}}
\newcommand{\pullbackbis}{\mbox{\LARGE{$\urcorner$}}}

\def\CCC{{\mathcal{C}}}
\def\TTT{{\mathcal{T}}}

\def\acc{^{\scriptstyle \emph{!`}}}

\def\hh{\mathrm{h}}
\def\pp{\mathrm{p}}
\def\ii{\mathrm{i}}

\newcommand{\Pac}{\PP^{\ac}}
\newcommand{\Pacc}{\PP^{\acc}}

\newcommand{\qi}{\xrightarrow{\sim}}

\newcommand{\coker}{\mathop{\mathrm{coker}}}

\newcommand{\Po}{\mathcal{P}}
\newcommand{\F}{\mathcal{F}}
\newcommand{\Fil}{\mathscr{F}}

\newcommand{\Y}{\vcenter{\xymatrix@M=0pt@R=6pt@C=6pt{
\ar@{-}[dr] &  &\ar@{-}[dl]  \\
 &\ar@{-}[d] &  \\  & &}}}
\newcommand{\YYY}{\vcenter{\xymatrix@M=0pt@R=6pt@C=6pt{
\ar@{-}[dr] &  &\ar@3{-}[dl]  \\
 &\ar@3{-}[d] &  \\  & &}}}
\newcommand{\cop}{\vcenter{\xymatrix@M=0pt@R=6pt@C=6pt{
 & \ar@{-}[d] & \\
 &\ar@{-}[dr] \ar@{-}[dl] &  \\  & &}}}

\newcommand{\copL}{\xymatrix@M=0pt@R=6pt@C=6pt{
 & \ar@{-}[d] & \\
 &\ar@{-}[dr] \ar@{-}[dl] &  \\  & &\\  & &\\  & &}}
\newcommand{\YL}{\vcenter{\xymatrix@M=0pt@R=6pt@C=6pt{
\ar@{-}[dr] &  &\ar@{-}[dl]  \\
 &\ar@{-}[d] &   \\  & &\\  & &\\  & &}}}
\newcommand{\YYL}{\vcenter{\xymatrix@M=0pt@R=6pt@C=6pt{
\ar@{-}[dr] &  &\ar@2{-}[dl]  \\
 &\ar@2{-}[d] &  \\  & &\\  & &\\  & &}}}
\newcommand{\LYY}{\vcenter{\xymatrix@M=0pt@R=6pt@C=6pt{
\ar@{-}[dr] &  &\ar@2{-}[dl]  \\
 &\ar@2{-}[d] &  \\  & &\\  & &\\  & &}}}

\newcommand{\XX}{\vcenter{\xymatrix@M=0pt@R=6pt@C=6pt{\ar@{-}[ddrr]&&\ar@{-}[ddll] \\ && \\ &&   }}}

\newcommand{\Ta}{\vcenter{\xymatrix@M=0pt@R=6pt@C=6pt{ \ar@{-}[dddrrr] && \ar@{-}[dl] &&  \\
&&& \ar@{-}[dl]  &  \\ &&&&  \ar@{-}[dl]  \\&&&  \ar@{-}[d] &
\\&&&& }}}
\newcommand{\Tb}{\vcenter{\xymatrix@M=0pt@R=6pt@C=6pt{  & \ar@{-}[dr]&&\ar@{-}[dl] \\
\ar@{-}[dr]&&\ar@{-}[dl]& \\&\ar@{-}[dr]&&\ar@{-}[dl]
\\&&\ar@{-}[d]& \\&&& }}}
\newcommand{\Tc}{\vcenter{\xymatrix@M=0pt@R=6pt@C=6pt{   \ar@{-}[dr]&&\ar@{-}[dl]& \\
&\ar@{-}[dr]&& \ar@{-}[dl] \\\ar@{-}[dr]&&\ar@{-}[dl]&
\\&\ar@{-}[d]&& \\&&& }}}
\newcommand{\Td}{\vcenter{\xymatrix@M=0pt@R=6pt@C=6pt{ && \ar@{-}[dr]&&\ar@{-}[dddlll] \\
 &\ar@{-}[dr]&&& \\ \ar@{-}[dr]&&&& \\& \ar@{-}[d]&&& \\&&&&  }}}
\newcommand{\Te}{\vcenter{\xymatrix@R=3pt@C=3pt{\ar@{-}[drdr] &&\ar@{-}[dl]  *=0{}
\ar@{-}[dr]&& \ar@{-}[ddll] \\ &&& *=0{}& \\&& *=0{} \ar@{-}[d]&&
\\&&&& }}}
\newcommand{\TaC}{\vcenter{\xymatrix@M=0pt@R=6pt@C=6pt{ \ar@{-}[ddddddrrrrrr] && \ar@{-}[dl] && && \\
&&& \ar@{-}[dl]  &&&  \\ &&&&  \ar@{-}[dl]&&  \\&&& &&&
\\&&&&\ar@{-}[dl]&& \\&&&&&\ar@{-}[dl]&\\&&&&&& }}}
\newcommand{\TreeL}{\vcenter{\xymatrix@M=0pt@R=5pt@C=5pt{ \ar@{-}[dr] &
&\ar@{-}[dl] & &  \\
& \ar@{-}[dr] & &\ar@{-}[dl]  & \\
& &\ar@{-}[d] & & \\
& & \\ & & }}}
\newcommand{\TreeR}{\vcenter{\xymatrix@M=0pt@R=5pt@C=5pt{
 & &\ar@{-}[dr] & & \ar@{-}[dl]  \\
& \ar@{-}[dr] & &\ar@{-}[dl]  & \\
& &\ar@{-}[d] & & \\
& & \\ & & }}}

\def\dd{\mathrm{d}}

\newcommand{\draftnote}[1]{}

\subjclass[2010]{Primary 18G55; Secondary 18D50}

\keywords{Homotopy algebra, operad, model category}

\thanks{This work was supported by the ANR grants HOGT and SAT}

\title{Homotopy theory of homotopy algebras}

\author{Bruno Vallette}

\dedicatory{In memoriam JLL}

\date{\today}

\address{Laboratoire Analyse, G\'eom\'etrie et Applications, Universit\'e Paris 13, Sorbonne Paris Cit\'e, CNRS, UMR 7539, 93430 Villetaneuse, France.}
\email{vallette@math.univ-paris13.fr}

\keywords{Homotopical algebra, model category, operad.}

\begin{document}

\maketitle

\begin{abstract}
This paper studies the homotopy theory of algebras and homotopy algebras over an operad. 
It provides an exhaustive description of their higher homotopical properties using the more general notion of morphisms called infinity-morphisms. 
The method consists in using the operadic calculus to endow the category of coalgebras over the Koszul dual cooperad or the bar construction with a new type of model category structure, Quillen equivalent to that of algebras. 
We provide an explicit homotopy equivalence for infinity-morphisms, which gives a simple description of the homotopy category, and we endow the category of homotopy algebras with an infinity-category structure. 
\end{abstract}

\tableofcontents

\section*{Introduction}
To define derived functors in  non-necessarily additive setting, D. Quillen generalized the ideas of A. Gro\-then\-dieck \cite{Grothendieck57} and introduced the notion of model category \cite{Quillen67}. A derived functor, being defined up to quasi-isomorphisms, it finds its source in the homotopy category, which is the original category localized with respect to quasi-isomorphisms. (This process is the categorical analogue of the construction of the field of rational numbers, where one starts from the ring of integers and formally introduce inverses for the non-zero numbers). For instance, Quillen homology theory for algebras of ``any" type is defined by deriving the functor of indecomposables, see \cite[Chapter~$12$]{LodayVallette12}. 

\smallskip

So it becomes crucial to be able to describe the homotopy category of algebras, and more generally the homotopy theory of algebras. Using the free algebra functor, Quillen explained how to transfer the cofibrantly generated model category of chain complexes to the category of differential graded algebras. His main theorem, then asserts that the homotopy category is equivalent to the full subcategory of fibrant-cofibrant objets, with morphisms up to a certain homotopy equivalence relation. In this model category structure \cite{Hinich97}, all the algebras are fibrant but the cofibrant ones are not so easily described: they are actually the retracts of quasi-free algebras on generators endowed with a suitable filtration. 

\smallskip

In his seminal paper on Rational Homotopy Theory \cite{Quillen69}, Quillen proved that several algebraic and topological homotopy categories are equivalent (differential graded Lie algebras, differential graded cocommutative coalgebras, topological spaces, simplicial spaces, etc.). 
For instance, one can find there a way to describe the homotopy category of differential graded Lie algebras
as the homotopy category of differential graded cocommutative coalgebras.
However, one \emph{problem} and one \emph{question} arise there.
The aforementioned Quillen equivalences hold only under a strong connectivity assumption; and why do the category of Lie algebras admits the ``dual'' category of cocommutative coalgebras?

\smallskip

The \emph{problem} was solved by Hinich \cite{Hinich01}, see also  Lefevre-Hasegawa \cite{LefevreHasegawa03}, who showed how to bypass the connectivity assumption by considering on cocommutative coalgebras a new class of weak equivalences, which is strictly included in the class of quasi-isomorphisms. The underlying idea is quite natural: the cobar functor going from differential graded cocommutative coalgebras to differential graded Lie algebras does not preserve quasi-isomorphisms. So if one wants this functor to form a Quillen equivalence, one has to find a class of morphisms of coalgebras which are sent to quasi-isomorphisms of Lie algebras. This forces the definition of weak equivalences of coalgebras. 

\smallskip

The \emph{question} is answered by the Koszul duality for operad \cite{GinzburgKapranov94, GetzlerJones94}. One encodes the category of Lie algebras with an operad and its Koszul dual cooperad is the one which encodes the category of cocommutative coalgebras. 

\smallskip

In the present paper, we settle  the homotopy theory of algebras over an operad as follows. 
First, we consider  the category of coalgebras over the Koszul dual cooperad, when the original operad is Koszul, or over the bar construction of the operad, in general. Then, we endow it with a new type of model category structure, Quillen equivalent to the one on algebras and  where the class of weak equivalences is strictly included in the class of quasi-isomorphisms. In this model category, all the coalgebras are cofibrant and the fibrant ones are the quasi-free ones, that is the ones for which the underlying coalgebra, forgetting the differential, is cofree. This already provides us with a simpler subcategory of fibrant-cofibrant objects. 

\begin{theointro} [\ref{theo:MCcoalg}]$\ $
The category of conilpotent dg $\Pacc$-coalgebras admits a model category structure, Quillen equivalent to that of dg $\PP$-algebras, such that every object is cofibrant and fibrant object are the quasi-free ones.  
 \end{theointro}

The method consists in extending Lefevre-Hasegawa's strategy from associative algebras to any algebra over an operad $\PP$, using heavily the operadic calculus developed in \cite[Chapters 10-11]{LodayVallette12}. (Notice that we try to provide complete proofs for all the results, like the form of fibrant objects, for instance, which was not given on the level of associative algebras).

\smallskip

Notice how the present theory parallels that of Joyal--Lurie of $\infty$-categories. To get a suitable notion of higher category, one can endow the category of simplicial sets with a new model category made up of less weak equivalences than the Quillen--Kan classical ones. In this case, all the objects are cofibrant  and the fibrant ones provide us with the notion of quasi-category, which is one model of  $\infty$-category. In our case, we consider a new model category on $\PP^{\ac}$-coalgebras with less weak equivalences. All the objects are cofibrant and the fibrant ones give us the suitable notion of $\PP_\infty$-algebra together with a good notion of morphisms that we call $\infty$-morphisms. 

\medskip 

\begin{center}
\begin{tabular}{|l|l|l|l|}
\hline
& {\sc Classical notion}\rule{0pt}{10pt} & 
{\sc Homotopy generalization}& {\sc Context}\\
\hline
Joyal--Lurie& Category \rule{0pt}{10pt} &
$\infty$-category& Simplicial sets \\
\hline
Present paper & $\PP$-algebras &$\PP_\infty$-algebras with $\infty$-morphisms &  $\PP^{\ac}$-coalgebras  \rule{0pt}{11pt}\\ 
\hline
\end{tabular}
\end{center}
\medskip

So this new model category structure induces a new homotopy theory for homotopy algebras with their $\infty$-morphisms. To complete the picture, we solve the problem of a functorial cylinder object inducing a universal homotopy relation for $\infty$-morphisms. (Notice that this question is not trivial, for instance, the author was not able to fix a crucial gap in \cite{BuijsMurillo13} related to this issue.)  The cylinder object that  we present here is proved to be equivalent to ``all'' the equivalence relations that have been considered so far  in the literature. 
With it, we prove that 
any $\infty$-quasi-isomorphism  admits a homotopy inverse, property which 
 does not hold for strict quasi-isomorphisms. 
Moreover, the initial category of algebras sits inside the category of homotopy algebras with their $\infty$-morphisms. And it can be proved that the second one retracts onto the first one. We are done:  the homotopy category of algebras is equivalent the following simple one. 

\begin{theointro} [\ref{theo:EquivalenceHoMAIN}] The following categories are equivalent
$$ 
\mathsf{Ho}(\mathsf{dg}\  \PP\textsf{-}\mathsf{alg}) \ \simeq\  
 \infty\textsf{-}\PP\textsf{-}\mathsf{alg}/\sim_{h} \ .$$ 
\end{theointro}

This gives a simple description of the first homotopical level of information about algebras over an operad. However, using the  simplicial localization methods of Dwyer--Kan \cite{DwyerKan80}, one can use the full power of this new model category structure on coalgebras to endow the category of homotopy algebras together with $\infty$-morphism with an $\infty$-category structure, thereby encoding all their higher homotopical information. 

\begin{theointro} [\ref{thm:inftycat}] 
The category $ \infty\textsf{-}\PP_\infty\textsf{-}\mathsf{alg}$ of $\PP_\infty$-algebra 
with $\infty$-morphisms extends to a simplicial category giving the same underlying homotopy category. 
\end{theointro}

Another direct corollary of the model category structure on coalgebras endows the category of homotopy algebras  with a fibrant objects category structure \cite{Brown73}. We reinforce this statement by proving that it actually carries a model category structure, except for the existence of some limits and colimits. In this case, the description of the three classes of structure maps is simple: weak equivalences (respectively cofibrations, respectively fibrations) are given by $\infty$-quasi-isomorphisms (respectively $\infty$-monomorphisms, respectively $\infty$-epimorphisms). For instance, this provides a neat description of fibrations between quasi-free coalgebras. 

\smallskip

There is already an extensive literature about model category structures on coalgebras over cooperads, see  \cite{Quillen67, Quillen69, GetzlerGoerss99, Hinich01, LefevreHasegawa03, AubryChataur03, Smith11, HessShipley14} for instance. The present result is more  general  for the  following three reasons. First, no assumption is needed here like bounded below chain complexes or finite dimensional space. Second, it treats the general case of any operad. Third, in most of the cases, the model category structures on coalgebras considers quasi-isomorphisms for weak equivalences. We show that such model category structures can be obtained from the present one by means of Bousfield localization. 

\smallskip 

Last but not least, let us mention the large range of applications of the present general homotopy theory. They include all the examples of algebraic structures treated in the compendium of Chapter~13 of \cite{LodayVallette12}. 
Therefore this applies to all the fields where homotopy algebras and $\infty$-morphisms play a role. 
So far, applications have been found, at least, in the following fields.

\begin{itemize}

\item[$\diamond$] {\sc $\infty$-morphisms of mixed complexes  ($D_\infty$-morphisms):}

\noindent cyclic homology, spectral sequences. 

\smallskip

\item[$\diamond$] {\sc $\infty$-morphisms of assocative algebras ($A_\infty$-morphisms):}

\noindent
algebraic topology (loop spaces, Massey products), symplectic geometry (Fukaya categories), probability theory (free probability).

\smallskip

\item[$\diamond$] {\sc $\infty$-morphisms of Lie algebras ($L_\infty$-morphisms):}

\noindent 
deformation theory, differential geometry (deformation quantization of Poisson manifolds).

\smallskip

\item[$\diamond$] {\sc $\infty$-morphisms of commutative algebras ($C_\infty$-morphisms):}

\noindent 
rational homotopy theory (K\"ahler manifolds).

\smallskip

\item[$\diamond$] {\sc $\infty$-morphisms of Gerstenhaber algebras ($G_\infty$-morphisms):}

\noindent 
Deligne conjecture, Drinfeld associators.

\smallskip

\item[$\diamond$] {\sc $\infty$-morphisms of Batalin--Vilkovisky algebras ($BV_\infty$-morphisms):}

\noindent 
Quantum cohomology (Frobenius manifolds), mirror symmetry.

\end{itemize}

In most of these cases, the homotopy theory of $\infty$-morphisms was possible to achieve ``by hands'' because the associated operad is rather small. This is not necessarily the case in the new appearing algebraic structures, like the homotopy Batalin--Vilkovisky algebras for instance. This point was the starting motivation for the development of the present general theory. 

\smallskip 

\noindent
\textbf{Layout.}  The paper is organized as follows. We begin with some recollections on operadic homological algebra and on the model category for algebras. In the second section, we endow the category of coalgebras over a Koszul dual cooperad with a new model category structure. The induced homotopy theory for $\infty$-morphisms is developed in Section~$3$.  The last section proves that there is  almost a model category structure  on homotopy algebras. Appendix~\ref{App:Obstruction} deals with the obstruction theory of $\infty$-morphisms and Appendix~\ref{app:TechLemma} contains the proof of a technical lemma. 

\smallskip 

\noindent
\textbf{Prerequisites.} The reader is supposed to be familiar with the notion of an operad and operadic homological algebra, for which  we refer to the book \cite{LodayVallette12}. In the present paper, we use the same notations as in loc. cit.. 

\smallskip

\noindent
\textbf{Framework.} Throughout this paper, we work over a field $\KK$ of characteristic $0$. Every chain complex is $\ZZ$-graded with homological degree convention, i.e. with degree $-1$ differential. 
 All the $\Sy$-modules $M=\lbrace M(n)\rbrace_{n\in\NN}$ are reduced, that is $M(0)=0$.

\section{Recollections}

In this section, we recall the main results about algebras over an operad \cite[Chapters~$10$--$11$]{LodayVallette12}, the (inhomogeneous) Koszul duality theory of operads \cite[Appendix~A]{GCTV12} and a (cofibrantly generated) model category structure for algebras \cite{Hinich97}.

\subsection{Operad and cooperad}

Recall that a  \emph{differential graded operad} (respectively a \emph{differential graded cooperad}) is a monoid (resp. a comonoid) in the monoidal category $(\textsf{dg}\ \Sy\textsf{-Mod}, \circ, \I) $ of differential graded $\Sy$-modules. 
A dg operad $\PP$ (resp. a dg cooperad $\CCC$)  is called \emph{augmented} (resp. \emph{coaugmented}) when it is equipped with an augmentation 
morphism $\epsilon \, : \, \PP \to \I$ of dg operads (resp. with a coaugmentation 
morphism $\eta \, :\, \I \to
\CCC $ of dg cooperads). 

\subsection{Algebra and coalgebra}

Let $\PP$ be a dg operad. A \emph{dg $\PP$-algebra} $(A, \gamma_A)$ is a left $\PP$-module concentrated in arity $0$: 
$$\gamma_A : \PP(A):=\PP \circ A=\bigoplus_{n\ge 1 } \PP(n)\t_{\Sy_{n}} A^{\t n} \to A\ . $$
A morphism of dg $\PP$-algebras is a morphism of dg modules $f : A\to B$ which commutes with the structure maps, i.e. $f \gamma_A=\gamma_B \PP(f)$.
This category is denoted by  $\textsf{dg} \ \PP\textsf{-alg}$. 

Dually, let $\CCC$ be a dg cooperad.  A \emph{conilpotent dg $\CCC$-coalgebra} $(C, \Delta_C)$ is a left $\CCC$-comodule concentrated in arity $0$: 
$$\Delta_C : C \to  \CCC(C):=\CCC \circ C=\bigoplus_{n\ge 1 } \CCC(n)\t_{\Sy_{n}} C^{\t n} \ . $$
This category is denoted  by $\textsf{conil}\ \textsf{dg} \ \CCC\textsf{-coalg}$.

In general, one defines the notion of a $\CCC$-coalgebra with the product and invariant elements:  $\Delta_C(C)\subset \prod_{n\ge 1 } (\CCC(n)\t C^{\t n})^{\Sy_{n}}$. Since we work over a field of characteristic $0$, we can identify invariant and coinvariant elements. We restrict here to the case where the image of the coproduct lies in the sum. We refer to \cite[Chapter~$5$]{LodayVallette12} for more details.

A dg $\PP$-algebra $A$ (resp. a dg $\CCC$-coalgebra $C$) is called \emph{quasi-free} when its underlying graded module, i.e. after forgetting the differential map, is isomorphic to a free $\PP$-algebra: $A\cong \PP(V)$ (resp. to a cofree $\CCC$-coalgebra: $C\cong \CCC(V)$). 

\subsection{Operadic homological algebra}\label{subsec:OHA}
Let $\PP$ be a dg operad and let $\CCC$ be a dg cooperad. The graded vector space of $\Sy$-equivariant maps 
$\Hom_\Sy(\CCC, \PP):=\prod_{n\ge 1} \Hom_{\Sy_n}(\CCC(n), \PP(n))$
from $\CCC$ to $\PP$ carries a natural dg Lie algebra structure, called the \emph{convolution Lie algebra}. An \emph{operadic twisting morphism} $\alpha \in \Tw(\CCC, \PP)$ is an element $\alpha : \CCC \to \PP$ of the convolution Lie algebra which satisfies the Maurer--Cartan equation: $$\partial \alpha + \frac{1}{2} [\alpha, \alpha]=0\ .$$

This operadic twisting morphism bifunctor is represented by the bar construction on the right-hand side and by the cobar construction on the left-hand side
$$\Hom_{\mathsf{dg} \ \mathsf{Op}}\left(\Omega \CCC,\, \PP\right) \cong
\mathrm{Tw}(\CCC,\, \PP) \cong  \Hom_{\mathsf{conil} \ \mathsf{dg} \ 
\mathsf{coOp}}\left(\CCC,\, \B \PP \right).$$
So these latter ones form a pair of adjoint functors. 
$$\B  \ : \    \textsf{augmented dg operads}  \ \rightleftharpoons\  
\textsf{conilpotent dg  cooperads}
   \ : \  \Omega \ . $$

Any twisting morphism $\alpha\in \Tw(\CCC, \PP)$ gives rise to twisted differentials $d_\alpha$ on the composite products $\CCC \circ \PP$ and $\PP\circ \CCC$. The resulting chain complexes are called \emph{left} and \emph{right twisting composite product} and are denoted $\CCC \circ_\alpha \PP$ and $\PP\circ_\alpha \CCC$ respectively, see \cite[Chapter~$6$]{LodayVallette12}.

\subsection{Bar and cobar constructions on the algebra level}\label{subsec:barcobarAlg}
Let  $\alpha \in \Tw(\CCC, \PP)$ be an operadic twisting morphism. 
 Let $(A, \gamma_A)$ be a dg
$\PP$-algebra and let $(C, \Delta_C)$ be a dg $\CCC$-coalgebra. We consider
the following unary operator $\star_\alpha$ of degree $-1$ on $\Hom(C,A)$:
$$\star_\alpha (\varphi) \ : \ C \xrightarrow{\Delta_C} \CCC \circ C
\xrightarrow{\alpha \circ \varphi} \PP\circ A
\xrightarrow{\gamma_A} A,\ \textrm{for}\ \varphi\in \Hom(C,A). $$ A \emph{twisting morphism with respect
to $\alpha$}
is a linear map $\varphi \, : \, C \to A$ of degree $0$ which is a
solution to the \emph{Maurer-Cartan} equation
$$\partial(\varphi) + \star_\alpha(\varphi)=0.$$ 
The space of twisting morphisms with respect to $\alpha$ is denoted by $\mathrm{Tw}_\alpha(C, A)$.

This twisting morphism bifunctor is represented by the the bar construction on the right-hand side and by the cobar construction on the left-hand side
$$\B_\alpha  \ : \    \textsf{dg} \ \PP\textsf{-alg}  \ \rightleftharpoons\  
\textsf{conil}\ \textsf{dg} \ \CCC\textsf{-coalg}
   \ : \  \Omega_\alpha \ . $$
So they form  a pair of adjoint functors. The underlying spaces are given by a cofree $\CCC$-coalgebra, $\B_\alpha A = \CCC(A)$, and by a free $\PP$-algebra, $\Omega_\alpha C =\PP(C)$, respectively. For more details, see \cite[Chapter~$11$]{LodayVallette12}.

\subsection{Koszul duality theory}
Let $(E,R)$ be a quadratic-linear data, that is $R \subset E \oplus \TTT(E)^{(2)}$. It gives rise to a quadratic operad $\PP:=\PP(E, R)=\TTT(E)/(R)$, where $\TTT(E)$ stands for the free operad on $E$ and where $(R)$ stands for the ideal generated by $R$. 

Let  $q:\TTT(E) \epi  \TTT(E)^{(2)}$ be the projection onto the quadratic part of the free operad. The image of $R$ under $q$, denoted $qR$, is homogeneous quadratic. So the associated quotient operad $q\PP:= \PP(E,qR)$ 
is homogeneous quadratic. We consider the homogeneous quadratic cooperad $q\PP^{\ac}:= \CCC(sE,s^2qR)$, where $s$ denotes the suspension map. 
 We assume that the space of relations $R$ satisfies the condition
$$(ql_1)\ : \ R\cap E =\{0\}\ , $$
which means  that the space of generators is minimal.  Under this assumption, there exists a map $\varphi :qR \to E$ such that $R$ is the graph of $\varphi$.
If $R$ satisfies the condition
$$(ql_2)\ : \  \{R\circ_{(1)} E + E \circ_{(1)} R \} \cap \TTT(E)^{(2)}  \subset  R\cap \TTT(E)^{(2)} \  , $$
which amounts to the maximality of the space of relations $R$, 
then the map $\varphi$ induces a square-zero coderivation $d_\varphi$ on the cooperad $q\PP^{\ac}$. For more details, we refer the reader to \cite[Appendix~A]{GCTV12} and to \cite[Chapter~$7$]{LodayVallette12}. From now on, we will always assume the two conditions $(ql_1)$ and $(ql_2)$. 

The dg cooperad $\PP^{\ac}:=(q\PP^{\ac}, d_\varphi)$ is called the \emph{Koszul dual  cooperad} of $\PP$. Notice that when the data $(E,R)$ is homogeneous quadratic, $qR=R$ and the differential map $d_\varphi$ vanishes. In this case, one recovers the homogeneous Koszul duality theory of Ginzburg--Kapranov and Getzler--Jones \cite{GinzburgKapranov94, GetzlerJones94}.

There is a canonical operadic twisting morphism $\kappa \in  \Tw(\PP^{\ac}, \PP)$ defined by the following composite 
$$ \kappa \, : \, \PP^{\ac}=\CCC(sE, s^2qR) \epi s E  \xrightarrow{s^{-1}} E \mono \PP\ . $$
The associated twisted composite product $\PP \circ_\kappa \Pac$ (resp. $\Pac \circ_\kappa \PP$) is called the \emph{Koszul complex}.

\begin{defi}$ \ $

\begin{enumerate}
\item A homogeneous quadratic operad is called a \emph{homogeneous Koszul operad} when its Koszul complex is acyclic. 

\item An quadratic-linear operad $\PP$ is called a \emph{Koszul operad} when its  presentation $\Po(E,R)$ satisfies conditions $(ql_1)$ and $(ql_2)$ and when the associated homogeneous quadratic operad $q\PP$ is homogeneous Koszul. 
\end{enumerate}
\end{defi}

When an inhomogeneous operad $\PP$ is Koszul, then its Koszul complexes $\PP \circ_\kappa \Pac$ and  $\Pac \circ_\kappa \PP$ are acyclic. 

\subsection{The general case}
For simplicity and to avoid discrepancy, the rest of the paper is written in the case of a Koszul presentation of the  operad $\Po$ using coalgebras over the Koszul dual cooperad $\Pac$. In general, one can always consider the trivial presentation made up of all the elements of $\Po$ as generators. This presentation is quadratic-linear and always Koszul. In this case, the Koszul dual cooperad is nothing but the bar construction $\B \Po$ of the operad $\Po$. So all the results of the present paper are always true if one considers this presentation and $\B \Po$-coalgebras. 

\subsection{Weight grading} Throughout the present paper, we will require an extra grading, other than the homological degree, to make all the proofs work. We work over the ground category of \emph{weight graded} dg $\Sy$-modules. This means that every dg $\Sy$-module is a direct sum of sub-dg $\Sy$-modules indexed by this weight $M_d = \bigoplus_{\omega \in \NN} M_d^{(\omega)}$, where $d$ stands for the homological degree and where  $\omega$  stands for the weight grading. In this context, 
the free operad is a weight graded dg operad, where the weight is given by the number of generators, or equivalently by the number of vertices under the tree representation. This induces a filtration on any quadratic operad $\PP=\TTT(E)/(R)$, where 
$$F_0 \PP=\I,\ F_1 \PP=\I \oplus E,\  \text{and}\ F_2 \PP=\I \oplus \frac{E\oplus \TTT(E)^{(2)}}{R}\ . $$
The underling cooperad of any Koszul dual cooperad is \emph{connected weight graded}: 
$$ q\Pac = \KK\, \I \oplus  q{\Pac}^{(1)} \oplus \cdots \oplus q{\Pac}^{(\omega)} \oplus \cdots =
\KK\, \I \oplus  sE \oplus s^2 R \oplus \cdots \ . $$
The coderivation $d_\varphi$ of the Koszul dual cooperad does not preserve this weight grading but lowers it by $1$: 
$$d_\varphi : q{\Pac}^{(\omega)} \to q{\Pac}^{(\omega-1)} \ . $$

\subsection{Homotopy algebra}
We denote by $\PP_\infty:=\Omega \PP^{\ac}$ the cobar construction of the Koszul dual cooperad of $\PP$. 
When the operad $\PP$ is a Koszul operad, this is a resolution $\PP_\infty \qi \PP$ of $\PP$.
Algebras over this dg operads are called \emph{$\PP_\infty$-algebras} or \emph{homotopy $\PP$-algebras}.

A $\PP_\infty$-algebra structure $\mu : \PP^{\ac} \to \End_A$ on a dg module $A$ is equivalently given by a square-zero coderivation $d_\mu$, called a codifferential, on the cofree $\PP^{\ac}$-coalgebra $\PP^{\ac}(A)$: 
$$
\Hom_{\mathsf{dg\ Op}}(\Omega\PP^{\ac},\, {\End}_A)  \cong  \mathop{\rm Codiff}(\PP^{\ac}(A))\ .$$ 
By definition, an \emph{$\infty$-morphism of $\PP_\infty$-algebras} is a morphism 
$$F \, : \big(\PP^{\ac}(A), d_\mu\big) \to \big(\PP^{\ac}(B), d_\nu\big)$$ of dg $\PP^{\ac}$-coalgebras.
The composite of two $\infty$-morphisms is defined as the composite of the associated morphisms of dg $\PP^{\ac}$-coalgebras:
$$F\circ G := \PP^{\ac}(A) \to \PP^{\ac}(B) \to \PP^{\ac}(C) \ .$$

The category of $\PP_\infty$-algebras with their $\infty$-morphisms  is denoted by $\infty\textsf{-}\PP_\infty\textsf{-alg}$. An $\infty$-morphism between $\PP_\infty$-algebras is denoted by 
$A \rightsquigarrow B$
to avoid confusion with the classical notion of morphism. 

Since an $\infty$-morphism $A \rightsquigarrow B$  is a morphism of $\Pac$-coalgebras $\PP^{\ac}(A) \to \PP^{\ac}(B)$, it is characterized by its projection $\PP^{\ac}(A) \to  B$ onto the space of generators $B$. The first component $A \cong \I(A) \subset \Pac(A) \to B$ of this map is a morphism of chain complexes. When this is a quasi-isomorphism (resp. an isomorphism), we say that the map $F$ is an \emph{$\infty$-quasi-isomorphism} (resp. an \emph{$\infty$-isomorphism}). One of the main property of $\infty$-quasi-isomorphisms, which does not hold for quasi-isomorphisms, lies in the following result. 
\begin{theo}[Theorem~$10.4.4$ of \cite{LodayVallette12}]\label{InverseInftyQI}
Let $\PP$ be a Koszul operad and let $A$ and $B$ be two $\PP_\infty$-algebras. If there exists an $\infty$-quasi-isomorphism $ A \stackrel{\sim}{\rightsquigarrow} B$, then there exists an $\infty$-quasi-isomorphism in the opposite direction $B\stackrel{\sim}{\rightsquigarrow} A$, which is the inverse of $H(A)\xrightarrow{\cong}H(B)$ on the level on holomogy. 
\end{theo}
So being $\infty$-quasi-isomorphic is an equivalence relation, which we call the \emph{homotopy equivalence}.
A complete treatment of the notion of  $\infty$-morphism is given in \cite[Chapter~$10$]{LodayVallette12}.

\subsection{The various categories}

We apply the arguments of Section~\ref{subsec:barcobarAlg} to the universal twisting morphism $\iota  \, :\, \Pac \to
\Omega \Pac=\PP_\infty$ and to the canonical twisting morphism $\kappa\, :\, \PP^{\ac} \to \PP$. This provides us with two bar-cobar adjunctions respectively.

By definition of $\infty$-morphisms, the bar construction 
$ \B_\iota :  \PP_\infty\textsf{-alg}  \to \textsf{conil}\ \textsf{dg} \ \Pac\textsf{-coalg}$ 
extends to a functor 
$ \widetilde{\B}_\iota :  \infty\textsf{-}\PP_\infty\textsf{-alg}  \to \textsf{conil}\ \textsf{dg} \ \Pac\textsf{-coalg}$. This latter one actually lands in quasi-free $\Pac$-coalgebras, yielding an isomorphism of categories. These various functors form the following diagram. 
$$\xymatrix@R=50pt@C=50pt@M=5pt{
\textsf{dg} \ \PP\textsf{-alg}\  \ar@<+0.5ex>@^{^{(}->}[d]^i  \ar@_{->}@<-0.5ex>[r]_(0.45){\B_\kappa}
&
 \ar@<-0.5ex>@_{->}[l]_(0.55){\Omega_\kappa} \ \textsf{conil}\  \textsf{dg} \  \PP^{\ac}\textsf{-coalg} \\ 
 \infty\textsf{-}\PP_\infty\textsf{-alg} \ar[r]^(0.43)\cong_(0.43){\widetilde{\B}_\iota} \ar@/^2pc/@{-^{>}}[u]^{\Omega_\kappa \widetilde{\B}_\iota} 
&
  \textsf{quasi}\textsf{-}\textsf{free}\  \PP^{\ac}\textsf{-}\textsf{coalg} \ar@{^{(}->}[u] }$$

\begin{theo}[Rectification \cite{GCTV12}]\label{theo:Rectification}Let $\PP$ be a Koszul operad.
\begin{enumerate}
\item  The functors
$$\Omega_\kappa \widetilde{\B}_\iota \ : \  \infty\textsf{-}\PP_\infty\textsf{-}\mathsf{alg}\  \rightleftharpoons \
\mathsf{dg}\ \PP\textsf{-}\mathsf{alg} \ : \ i$$ 
form a pair of adjoint functors, where $i$ is right adjoint to $\Omega_\kappa \widetilde{\B}_\iota$.  

\item Any homotopy $\PP$-algebra $A$ is naturally $\infty$-quasi-isomorphic to the dg $\PP$-algebra $\Omega_\kappa \widetilde{\B}_\iota A$:
$$A \stackrel{\sim}{\rightsquigarrow}   \Omega_\kappa \widetilde{\B}_\iota A\ .$$
\end{enumerate}
\end{theo}

The dg $\PP$-algebra $\Omega_\kappa \widetilde{\B}_\iota A$, homotopically equivalent to the $\PP_\infty$-algebra $A$ is called the \emph{rectified} $\PP$-algebra.  We refer the reader to \cite[Chapter~$11$]{LodayVallette12} for more details. 

\subsection{Homotopy categories}
Recall that the \emph{homotopy category} $\textsf{Ho}(\mathsf{dg}\ \PP\textsf{-alg})$ (resp. 
$\textsf{Ho}(\infty\textsf{-}\PP_\infty\textsf{-alg})$) 
is  the localization of 
the category of dg $\PP$-algebras with respect to the class of quasi-isomorphisms 
 (resp. $\PP_\infty$-algebras with respect to the class of $\infty$-quasi-isomorphisms). 
The rectification adjunction of Theorem~\ref{theo:Rectification} induces an equivalence of categories between  these two homotopy categories. 

\begin{theo}[Theorem~$11.4.8$ of \cite{LodayVallette12}]\label{theo:HoEquiv}
Let $\PP$ be a Kos\-zul operad.  The homotopy category of dg $\PP$-algebras and the homotopy category of $\PP_\infty$-algebras with the $\infty$-morphisms are equivalent
$$\mathsf{Ho}(\mathsf{dg}\ \PP\textsf{-}\mathsf{alg}) \cong \mathsf{Ho}(\infty\textsf{-}\PP_\infty\textsf{-}\mathsf{alg})\ .$$
\end{theo}

\subsection{Model category for algebras}

A \emph{model category structure} consists of the data of three distinguished classes of maps: \emph{weak equivalences}, \emph{fibrations} and \emph{cofibration}, subject to five axioms. This extra data provided by fibrations and cofibrations gives a way to describe the homotopy category, defined by localization with respect to the weak equivalences. This notion is due to D. Quillen \cite{Quillen67}; we refer the reader to the reference book of M. Hovey \cite{Hovey99} for a comprehensive presentation.

\begin{theo}[\cite{Hinich97}]\label{theo:Hinich97}
The following classes of morphisms endow the  category of dg $\PP$-algebras with a model category structure.

\begin{itemize}
\item[$\diamond$] The class $\mathfrak W$ of \emph{weak equivalences} is given by the quasi-isomorphisms; 

\item[$\diamond$] the class $\mathfrak F$ of \emph{fibrations} is given by degreewise epimorphisms, $f_n : A_n \epi B_n$ ; 

\item[$\diamond$] the class $\mathfrak C$ of \emph{cofibrations} is given by the maps which satisfy the left lifting property with respect to acyclic fibrations $\mathfrak F \cap \mathfrak W$. 
\end{itemize}
\end{theo}

Notice that this model category structure is cofibrantly generated since it is obtained by transferring the cofibrantly generated model category structure on dg modules thought the free $\PP$-algebra functor, which   is left adjoint to the forgetful functor, see \cite[Section~$\textrm{II}.4$]{Quillen67}. 

Following D. Sullivan \cite{Sullivan77}, we call \emph{triangulated dg $\PP$-algebra} any quasi-free dg  $\PP$-algebra $(\PP(V), d)$ equipped with an exhaustive filtration 
$$V_0=\lbrace 0 \rbrace \subset V_1 \subset V_2 \subset \cdots \subset \textrm{Colim}_i V_i =V \ ,  $$
satisfying $d(V_i)\subset \PP(V_{i-1})$.

\begin{prop}[\cite{Hinich97}]
With respect to the aforementioned model category structure, every dg $\PP$-algebra is fibrant and a dg $\PP$-algebra is cofibrant if and only if its a retract of a triangulated dg $\PP$-algebra. 
\end{prop}

So this model category structure is right proper.

\section{Model category structure for coalgebras}

For a Koszul operad $\Po$,  we endow the category of conilpotent dg $\PP^{\ac}$-coalgebras with a model category structure which makes the bar-cobar adjunction a Quillen equivalence with dg $\Po$-algebras. 

\subsection{Main theorem}

\begin{defi}
In the category of conilpotent dg $\PP^{\ac}$-coalgebras, we consider the following three classes of morphisms. 

\begin{itemize}
\item[$\diamond$] The class $\mathfrak W$ of \emph{weak equivalences} is given by the morphisms of dg $\Pac$-coalgebras $f : C \to D$ whose image $\Omega_\kappa f : \Omega_\kappa C \qi \Omega_\kappa D$  under the cobar construction is a quasi-isomorphism of dg $\PP$-algebras;
\item[$\diamond$] the class $\mathfrak C$ of \emph{cofibrations} is given  by degreewise monomorphisms, $f_n : C_n \mono D_n$ ; 
\item[$\diamond$] the class $\mathfrak F$ of \emph{fibrations} 
is given by the maps which satisfy the right lifting property with respect to acyclic cofibrations $\mathfrak C \cap \mathfrak W$. 
\end{itemize}
\end{defi}

\begin{theo}\label{theo:MCcoalg}\leavevmode 

\begin{enumerate}
\item Let $\PP$ be a Koszul operad. 
The aforementioned three classes  of morphisms  form a model category structure on conilpotent dg $\Pacc$-coalgebras. 

\item With this model category structure, every conilpotent dg $\Pacc$-coalgebra is cofibrant; so this model category is left proper. A conilpotent dg $\Pacc$-coalgebra is fibrant if and only  if it is isomorphic to a 
quasi-free dg $\Pacc$-coalgebra.

\item The bar-cobar adjunction 
$$\Bk  \ : \    \mathsf{dg} \ \PP\textsf{-}\mathsf{alg}  \ \rightleftharpoons\  
\mathsf{conil}\ \mathsf{dg} \ \PP^{{^{\scriptstyle \emph{!`}}}}\textsf{-}\mathsf{coalg}
   \ : \  \Ok \ . $$
is a Quillen equivalence. 
\end{enumerate}

\end{theo}

\subsection{Weight filtration}

\begin{defi}
Any $\Pac$-coalgebra $(C, \Delta_C)$ admits the following \emph{weight filtration}: 
$$F_n C :=\left\lbrace c\in C\,  | \, \Delta_C(c) \in \bigoplus_{\omega =0}^n {\Pac}^{(\omega)}(C) \right\rbrace \ . $$
For instance, the first terms are 
$$F_{-1} C =\lbrace 0 \rbrace \subset F_0 C :=\left\lbrace c\in C\, |\,  \Delta_C(c)=c \right\rbrace
\subset
F_1 C :=\left\lbrace c\in C\, |\,  \Delta_C(c) \in C \oplus {\Pac}^{(1)}(C) \right \rbrace \subset \cdots \ . 
$$
\end{defi}

\begin{examples}$ \ $

\begin{itemize}
\item[$\diamond$] For any cofree coalgebra $\Pac (V)$, the weight filtration is equal to $F_n \Pac(V)=\bigoplus_{\omega =0}^n {\Pac}^{(\omega)}(V) $.

\item[$\diamond$]
When the operad $\PP$ is the operad $As$, which encodes associative algebras, the Koszul dual cooperad $As^{\ac}=As^c$ encodes coassociative coalgebras. In this case, the weight filtration is equal, up to a shift of indices, to the coradical filtration of coassociative coalgebras, cf. \cite[Appendix~B]{Quillen69} and \cite[Section~$1.2.4$]{LodayVallette12}. 
\end{itemize}
\end{examples}

We consider the reduced coproduct $\bar\Delta_C(c):=\Delta_C(c)-c$. Its kernel is equal to $\mathrm{Prim}\,  C:=F_0 C$, which is called the \emph{space of primitive elements}. 

\begin{prop}\label{prop:WeightFiltr}
Let $C$ be a conilpotent dg $\Pacc$-coalgebra $(C, d_c, \Delta_C)$. 

\begin{enumerate}
\item Its weight filtration is exhaustive: $\bigcup_{n\in \NN} F_n C = C$. 

\item Its weight filtration satisfies
$$\bar\Delta_C(F_n C)\subset \bigoplus_{1\leq \omega \leq n,\  k\ge 1, \atop n_1+\cdots+n_k=n-\omega} 
{\Pac}^{(\omega)}(k) \t_{\Sy_k}(F_{n_1} C\t  \cdots \t F_{n_k} C)\ . $$

\item The differential preserves the weight filtration:  $d_c(F_n C)\subset F_n C$.
\end{enumerate}
\end{prop}

\begin{proof}
The first point follows from the definition of a conilpotent coalgebra and from the fact that $\Pac$ is a connected weight graded cooperad. The second point is a direct corollary of the relation $\Pac(\Delta_C) \Delta_C = \Delta_{\Pac} (C)  \Delta_C  $ in the definition of a coalgebra over a cooperad. The last point is a consequence of  the commutativity of the differential $d_C$ and the structure map $\Delta_C$.
\end{proof}

This proposition shows that the weight filtration is made up of  dg $\Pac$-subcoalgebras. Notice that any morphism $f : C \to D$ of $\Pac$-coalgebras preserves the weight filtration: $f(F_n C)\subset F_n D$.

\subsection{Filtered quasi-isomorphisms}
In this section, we refine the results of \cite[Chapter~$11$]{LodayVallette12} and of \cite[Section~$1.3$]{LefevreHasegawa03} about the 
behavior of the bar and cobar constructions with respect to 
quasi-isomorphisms. 

\begin{defi}
A \emph{filtered quasi-isomorphism} of conilpotent dg $\Pac$-coalgebras is a morphism $f : C \to D$ of dg $\Pac$-coalgebras such that the induced morphisms of chain complexes 
$$\gr_n  f : F_n C/F_{n-1} C \qi F_n D/F_{n-1} D$$ are quasi-isomorphisms, for any $n\ge 0$.
\end{defi}

\begin{prop}\label{lemm:FiltQIisWE}
The class of filtered quasi-isomorphisms of conilpotent dg $\Pacc$-coalgebras is included in the class of weak equivalences. 
\end{prop}

\begin{proof}
Let $f : C \to D$ be a filtered quasi-isomorphism of conilpotent dg $\Pac$-coalgebras. We consider the following filtration on the cobar construction $\Omega_\kappa C = (\PP(C), d_1+d_2)$  induced by the weight filtration:
$$\F_n \, \Omega_\kappa C := \sum_{k\ge 1,  \atop n_1+\cdots+n_k \leq n} \PP(k)\t_{\Sy_k} (F_{n_1} C\t  \cdots \t F_{n_k} C) \ .$$
Recall from \cite[Section~$11.2$]{LodayVallette12}
that the differential of the cobar construction is made up of two terms $d_1+d_2$, where $d_1=\PP\circ' d_C$ and where $d_2$ is the unique derivation 
which extends
$$C \xrightarrow{\Delta_C} \Pac \circ C \xrightarrow{\kappa\circ \Id_C} \PP \circ C\ .$$
It is explicitly given by
\begin{align*}
\PP\circ C \xrightarrow{\Id_\PP \circ' \Delta_C} \PP \circ
(C; \Pac\circ C)\xrightarrow{\Id_\PP \circ
(\Id_C;   \kappa\circ \Id_C)} & \PP \circ
(C; \PP\circ C) 
 \cong & (\PP \circ_{(1)}\PP)(C)
\xrightarrow{\gamma_{(1)} \circ \Id_C} \PP \circ C.
\end{align*}
So, Proposition~\ref{prop:WeightFiltr} implies
$$d_1(\F_n) \subset  \F_n \quad \text{and} \quad d_2(\F_n) \subset  \F_{n-1}\ . $$
The first page of the associated spectral sequence is equal to 
$$\PP(\gr\, f) \ : \ (\PP(\gr\, C), \PP\circ' d_{\gr\, C})
\qi  (\PP(\gr\,  D), \PP\circ' d_{\gr D})\ , $$
which is a quasi-isomorphism by assumption. Since the weight filtration is exhaustive, this filtration is exhaustive. It is 
also bounded below, so we conclude by the classical  convergence theorem of spectral sequences \cite[Chapter~$11$]{MacLane95}: 
$$\Omega_\kappa f \ : \ \Omega_\kappa C  \qi \Omega_\kappa D \ . $$
\end{proof}

\begin{prop}\label{lemm:QIdonneFilteredQI}
If $f : A \qi A'$ is a quasi-isomorphism of dg $\PP$-algebras, then $B_\kappa f : \B_\kappa A \to \B_\kappa A'$ is a filtered quasi-isomorphism of conilpotent dg $\Pac$-coalgebras.
\end{prop}

\begin{proof}
Since the bar construction $\B_\kappa A$ is a quasi-free coalgebra, its weight filtration is equal to 
$F_n\,  \B_\kappa A=\bigoplus_{\omega =0}^n {\Pac}^{(\omega)}(A)$. Recall that its differential is the sum of three terms $d_\varphi \circ \Id_A + \Id_{\Pac} \circ ' d_A + d_2$, where $d_2$ is the unique coderivation which extends 
$$\Pac\circ A \xrightarrow{\kappa\circ \Id_A} \PP \circ A \xrightarrow{\gamma_A} A \ . $$
So, the coderivation $d_2$ is equal
to the composite
\begin{align*}
\Pac\circ A \xrightarrow{\Delta_{(1)}\circ \Id_A}(\Pac
\circ_{(1)} \Pac) \circ A  \xrightarrow{(\Id_{\Pac}  \circ_{(1)}
\kappa)\circ \Id_A} &\ (\Pac \circ_{(1)} \PP)\circ A \\
 \cong &\ \Pac
\circ(A; \PP\circ A) \xrightarrow{\Id_{\Pac} \circ(\Id_A; \gamma_A)}
 \Pac \circ A\ .
\end{align*}
Since the maps $\kappa$ and $d_\varphi$ lowers the weight grading by $1$, we get 
$$\Id_{\Pac} \circ ' d_A(F_n) \subset F_n, \quad d_\varphi \circ \Id_A (F_n) \subset F_{n-1}, \quad \text{and} \quad  
d_2(F_n)\subset F_{n-1}   \ . $$
Hence, the graded analogue of $\B_\kappa f$ is equal to 
$$\gr_n \, \B_\kappa f={\Pac}^{(n)}(f) : \big({\Pac}^{(n)}(A), \Id_{\Pac} \circ ' d_A\big) \qi \big({\Pac}^{(n)}(A'), \Id_{\Pac} \circ ' d_{A'}\big) \ ,   $$
which is a quasi-isomorphism for any $n\in \NN$.
\end{proof}

\begin{prop}\label{prop:weINqi}
The class of weak equivalences  of conilpotent dg $\Pacc$-coalgebras is included in the class of quasi-isomorphisms.
\end{prop}

\begin{proof}
 Let $f : C \to D$ be a weak equivalence of conilpotent dg $\Pac$-coalgebras. By definition, its image under the bar construction $\Ok f : \Ok C  \xrightarrow{\sim} \Ok D$ is a quasi-isomorphism of dg $\PP$-algebras. Since the bar construction $\Bk$ preserves quasi-isomorphisms, by \cite[Proposition~$11.2.3$]{LodayVallette12}, and since 
the counit of the  bar-cobar adjunction 
$\upsilon_\kappa C \, :\, C \xrightarrow{\sim} \B_\kappa \Omega_\kappa C$
 is a quasi-isomorphism, we conclude with the following commutative diagram 
$$\xymatrix@R=30pt@C=30pt@M=5pt{C \ar[r]^f \ar[d]^\sim_{\upsilon_\kappa C} & D \   \ar[d]^\sim_{\upsilon_\kappa D}  \\
\Bk\Ok C \ar[r]^\sim   & \Bk\Ok D\ . }$$
\end{proof}

Without any assumption on the connectivity of the underlying chain complexes, the aforementioned inclusion can be strict. Examples of quasi-isomorphisms, which is not  weak equivalences, are given 
in \cite[Section~$9.1.2$]{Hinich01} of dg cocommutative coalgebras and
in \cite[Section~$1.3.5$]{LefevreHasegawa03} of dg coassociative coalgebras.\\

The following diagram sums up the aforementioned propositions. 
$$
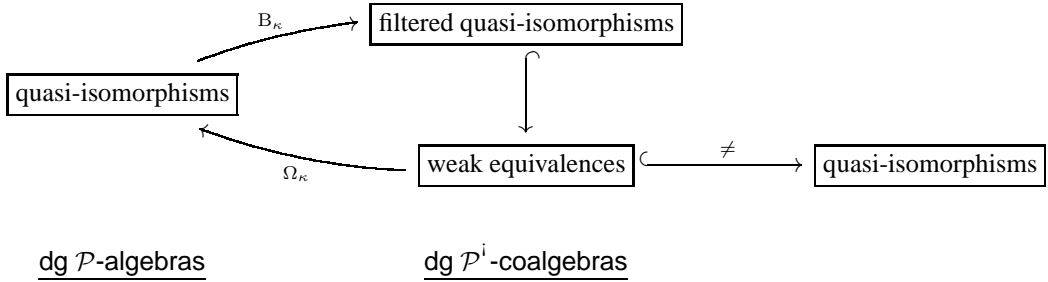
\xymatrix@M=5pt@C=40pt@R=1pt{ 
& \boxed{\text{filtered quasi-isomorphisms}}  \ar@{^{(}->}[dd] &
\\
\boxed{\text{quasi-isomorphisms}}   \ar@/^1pc/[ru]^-{\B_\kappa} & &\\
 & \ar@/^1pc/[ul]^-{\Omega_\kappa} \boxed{\text{weak equivalences}} \ar@{^{(}->}[r]^{\neq} & 
 \boxed{\text{quasi-isomorphisms}}  \\
\\ \underline{\textsf{dg} \ \PP\textsf{-algebras}} & \underline{\textsf{dg} \  \Pac\textsf{-coalgebras}}&
 }  $$

\begin{theo}\label{Thm:KoszulBarCobarRes}
Let $\PP$ be a Koszul operad. 
\begin{enumerate}
\item   The counit of the bar-cobar adjunction $\varepsilon_\kappa \,: \,
\Omega_\kappa \B_\kappa A \xrightarrow{\sim} A $ is a
quasi-isomorphism of dg $\PP$-algebras, for every dg $\PP$-algebra $A$. 

\item 
The unit of the bar-cobar adjunction 
$\upsilon_\kappa \, :\, C \xrightarrow{\sim} \B_\kappa \Omega_\kappa C$ is a  weak equivalence of conilpotent dg $\Pac$-coalgebras, for every conilpotent dg $\Pacc$-coalgebra $C$.
\end{enumerate}
\end{theo}

\begin{proof}
The first point follows from \cite[Theorem~$11.3.3$ and Corollary~$11.3.5$]{LodayVallette12}. 
For the second point, we consider the following filtration induced by the weight filtration of $C$: 
$$\F_n \, \B_\kappa \Omega_\kappa C := \sum_{k\ge 1,  \atop n_1+\cdots+n_k \leq n} (\Pac\circ \PP)(k)\t_{\Sy_k} (F_{n_1} C\t  \cdots \t F_{n_k} C) \ .$$
Since the unit of adjunction $\upsilon_\kappa$ is equal to the composite
$$ C \xrightarrow{\Delta_C} \Pac(C) \cong \Pac\circ \I \circ C \mono \Pac\circ \PP \circ C \ ,
$$
it preserves the respective filtrations by Proposition~\ref{prop:WeightFiltr}--$(2)$. The associated graded morphism is equal to 
$$\gr_n  \upsilon_\kappa : F_n C/F_{n-1} C \to \F_n \B_\kappa \Omega_\kappa C/\F_{n-1} \B_\kappa \Omega_\kappa C \ ,$$
where the right-hand side is isomorphic to $\B_\kappa \Omega_\kappa \, \gr\,  C\cong (\Pac \circ_\kappa \PP)  \circ (\gr\,  C, d_{\gr\,  C})$. Since the operad $\PP$ is Koszul, its Koszul complex is acyclic, $\Pac \circ_\kappa \PP\qi \I$, which proves that the unit $\upsilon_\kappa$ is a filtered quasi-isomorphism. 
Finally, we conclude that the unit $\upsilon_\kappa$ is a weak-equivalence by  Proposition~\ref{lemm:FiltQIisWE}. 
\end{proof}

\subsection{Fibrations and cofibrations}\label{subsec:Fib-Cofib}

Let us first recall that the coproduct $A\vee B$ of two $\Po$-algebras $(A,\gamma_A)$ and $(B, \gamma_B)$ is given by the following coequalizer 
$$\xymatrix{\Po\circ (A\oplus B; \Po(A)\oplus \Po(B))    \ar@<0.5ex>[r] \ar@<-0.5ex>[r] &   \Po(A\oplus B) \ar@{->>}[r] & A\vee B   \ , } $$
where one map is induced by the partial composition product $\Po\circ_{(1)} \Po \to \Po$ of the operad $\Po$
and where the other one is equal to 
$\Po\circ(\Id_{A\oplus B}; \gamma_A+\gamma_B)$. When $B=\Po(V)$ is a free $\Po$-algebra, the coproduct $A\vee \Po(V)$ is simply equal to the following coequalizer
$$\xymatrix{\Po\circ (A\oplus V; \Po(A))   \ar@<0.5ex>[r] \ar@<-0.5ex>[r] &   \Po(A\oplus V) \ar@{->>}[r] & A\vee \Po(V)\ . } $$

As usual, see \cite{Hovey99}, we denote by $D^n$ the acyclic chain complex 
$$ \cdots \to 0 \to \KK \xrightarrow{\sim} \KK \to 0 \to \cdots $$ 
concentrated in degrees $n$ and $n-1$. We denote by $S^n$ the chain complex 
$$ \cdots \to 0 \to \KK  \to 0 \to \cdots $$ 
concentrated in degrees $n$. The generating cofibrations of the model category of dg modules 
are the embeddings $I^n : S^{n-1} \mono D^n$ and the generating acyclic cofibrations are the quasi-isomorphisms $J^n : 0 \stackrel{\sim}{\mono} D^n$. So, in the cofibrantly generated model category of dg $\Po$-algebras, the relative $\Po(I)$-cell complexes, also known as \emph{standard cofibrations}, are the sequential colimits of pushouts of coproducts of $\Po(I)$-maps. Since we are working over a field $\KK$, such a pushout is equivalent to 
$$\xymatrix@M=5pt{\Po(s^{-1}V)\ar@{>->}[d]  \ar[r]^(0.55){\gamma_A \Po(s \alpha)}  
\ar@{}[dr] | (0.65)\pushout
& A \ar@{>->}[d] \\
\Po(V \oplus s^{-1}V) \ar[r]  & A \vee_\alpha \Po(V) \ ,    } $$
where $V$ is a graded module, $\alpha : V \to A$ is a degree $-1$ map, with image in the cycles of $A$. The dg $\Po$-algebra $A \vee_\alpha \Po(V)$ is equal to the coproduct of $\Po$-algebras $A\vee \Po(V)$ endowed with the differential given by $d_A$ and by the unique derivation which extends  the map 
$$ V \xrightarrow{\alpha} A \mono A \vee \Po(V)\ . $$
Hence a standard cofibration of dg $\Po$-algebras is a morphism of dg $\Po$-algebras $A \mono (A \vee \Po(S),d)$, where the graded module $S$ admits an exhaustive filtration 
$$S_0=\lbrace 0 \rbrace \subset S_1 \subset S_2 \subset \cdots \subset \textrm{Colim}_i \, S_i =S \ ,  $$
satisfying $d(S_i)\subset A\vee\PP(S_{i-1})$.

In the same way, a \emph{standard acyclic cofibration}, or relative $\Po(J)$-cell complex, is a morphism of dg $\Po$-algebras $A \mono A \vee \Po(M)$, where the chain complex  $M$ is a direct sum 
$M=\oplus_{i=1}^\infty M^i$
of acyclic chain complexes. Finally, any cofibration (resp. acyclic cofibration) is a retract of a standard cofibration (resp.  standard acyclic cofibration) with isomorphisms on domains. \\

Let $(A, d_A)$ be a dg $\Po$-algebra and let $(V, d_V)$ be a chain complex. Let $\alpha : V \to A$ be a degree $-1$ map such that the unique derivation on the coproduct $A\vee \Po(V)$, defined by $d_A$, $d_V$ and $\alpha$, squares to $0$. In this case, the dg $\Po$-algebra produced is still denoted by  $A\vee_\alpha \Po(V)$.

\begin{lemm}\label{lemma:attachdgmod}
The embedding $A \mono A \vee_\alpha \Po(V)$ is a standard cofibration of dg $\Po$-algebras. 
\end{lemm}

\begin{proof}
Since we are working over a field $\KK$, any chain complex $V$ decomposes into 
$$V \cong B \oplus H \oplus s B \ , $$
where $d_V(B)=d_V(H)=0$ and where $d_V : sB \xrightarrow{s^{-1}} B$. It is enough to consider the following filtration to conclude 
$$S_0=\lbrace 0 \rbrace \subset S_1:=B \oplus H \subset S_2:= V\ .  $$
\end{proof}

\begin{prop}\label{prop:MonodonneCof}
Let $(C,\Delta_C)$ be a dg $\Pacc$-coalgebra and let $C'\subset C$ be a dg sub-$\Pac$-coalgebra such that $\bar \Delta_C (C)\subset \Pacc(C')$. The image  of the inclusion $C' \mono C$ under the cobar construction $\Omega_\kappa$ is a standard cofibration of dg $\Po$-algebras. 
\end{prop}

\begin{proof}
Since we are working over a field $\KK$, there exists a graded sub-module $E$ of $C$ such that 
$C \cong C'\oplus E$ in the category of graded $\KK$-modules. Forgetting the differentials, the underlying $\Po$-algebra of the cobar construction of $C$ is isomorphic to 
$$\Omega_\kappa C \cong \PP (C' \oplus E)\cong \Po(C') \vee \Po(E) \ .   $$
Under the decomposition $C \cong C'\oplus E$, the differential $d_C$ of $C$ is the sum of the following three terms: 
$$d_{C'} : C' \to C', \quad d_E : E \to E, \quad \text{and} \quad \alpha : E \to C'  \ .  $$
By  assumption, the degree $-1$ map 
$$\beta : E \mono C \xrightarrow{\Delta_C} \Pac(C) \xrightarrow{\kappa(C)} \Po(C)$$ 
actually lands in $\Po(C')$. So the morphism of dg $\PP$-algebras $\Omega_\kappa C' \to \Omega_\kappa C$ is equal to the embedding $A \mono A \vee_{\alpha+\beta} \Po(E)$, where $A$ stands for the dg $\Po$-algebra $\Omega_\kappa C'$. We conclude the present proof with Lemma~\ref{lemma:attachdgmod}.
\end{proof}

\begin{theo}\label{theo:BarCobarFirCofibr}$ \ $ 

\begin{enumerate}
\item The cobar construction $\Omega_\kappa$ preserves cofibrations and weak equivalences.

\item The bar construction $\B_\kappa$ preserves fibrations and weak equivalences.
\end{enumerate}
\end{theo}

\begin{proof}$ \ $

\begin{enumerate}
\item Let $f : C \mono D$ be a cofibration of conilpotent dg $\Pac$-coalgebras. For any $n\in \NN$, we consider the dg sub-$\Pac$-coalgebra of $D$ defined by 
$$D^{[n]}:=f(C)+ F_{n-1} D \ ,$$
where $F_{n-1} D$ stands for the weight filtration of the $\Pac$-coalgebra $D$. By convention, we set  $D^{[0]}:=C$. Proposition~\ref{prop:WeightFiltr}--$(2)$ implies $\bar\Delta_{D^{[n+1]}}(D^{[n+1]}) \subset \Pac(D^{[n]})$. So,  we can apply Proposition~\ref{prop:MonodonneCof} to show that the maps $\Omega_\kappa D^{[n]} \to \Omega_\kappa D^{[n+1]}$ are standard cofibrations of dg $\Po$-algebras. Finally, the map 
$\Omega_\kappa  f$ is a cofibration as a  sequential colimit of standard cofibrations.

\noindent
The cobar construction $\Omega_\kappa$ preserves weak equivalences by definition. 

\item Let $g : A \epi A'$ be a fibration of dg $\Po$-algebras. Its image $\B_\kappa\,   g$ is a fibration if and only if it satisfies the right lifting property with respect to any acyclic cofibration $f : C \stackrel{\sim}{\mono} D$. Under the bar-cobar adjunction~\ref{subsec:barcobarAlg}, this property is equivalent to the left lifting property of $\Omega_\kappa f$ with respect to $g$, which holds true by the above point $(1)$.

\noindent 
The bar construction $\B_\kappa$ sends quasi-isomorphisms of dg $\Po$-algebras to weak equivalences of dg $\Pac$-coalgebras by Propositions~\ref{lemm:FiltQIisWE} and \ref{lemm:QIdonneFilteredQI}.
\end{enumerate}
\end{proof}

\subsection{Proof of Theorem~\ref{theo:MCcoalg}--$(1)$}

\begin{proof}(of Theorem~\ref{theo:MCcoalg}--$(1)$)

\begin{itemize}
\item[$(\text{MC}\, 1)$] [\emph{Limits and Colimits}] Since we are working over a field of characteristic $0$, Proposition~$1.20$ of \cite{GetzlerJones94} applies and shows that 
the category of conilpotent dg $\Pac$-coalgebras admits finite limits and finite colimits.

\item[$(\text{MC}\, 2)$] [\emph{Two out of three}] Let $f : C \to D$ and $g : D \to E$ be two morphisms of conilpotent dg $\Pac$-coalgebras. If any two of $f$, $g$ and $gf$ are weak equivalences, then so is the third. This is a direct consequence of the definition of weak equivalences and the axiom $(\text{MC}\, 2)$ for dg $\Po$-algebras.

\item[$(\text{MC}\, 3)$] [\emph{Retracts}] Since the cofibrations $\mathfrak{C}$ are the degreewise monomorphisms, they are stable under retracts. 

\noindent
Since the image of a retract under the cobar construction $\Omega_\kappa$ is again a retract, weak equivalences $\mathfrak{W}$ of conilpotent dg $\Pac$-coalgebras are stable under retract by the axiom $(\text{MC}\, 3)$ for dg $\Po$-algebras.

\noindent
Let $f : C \epi D$ be a fibration $\mathfrak{F}$ of conilpotent dg $\Pac$-coalgebras and let $g : E \to F$ be a retract of $f$. Let $c : G \stackrel{\sim}{\mono} H$ be an acyclic cofibration fitting into the following commutative diagram. 
$$\xymatrix@M=5pt@R=30pt@C=30pt{G \ar[r]  \ar@{>->}[d]_(0.45)c^(0.45)\sim&   E \ar[r] \ar[d]_(0.45)g  &   C \ar[r]^p \ar@{->>}[d]_(0.45)f &    E\ar[d]_(0.45)g \\
H  \ar[r]  \ar@{..>}[rru]^(0.7){\alpha}  | \hole  &    F   \ar[r]&    D  \ar[r] &  F     } $$
By the lifting property, there exists a map $\alpha : H \to C$ making the first rectangle into a commutative diagram. Finally, the composite $p \alpha$ makes the first square into a commutative diagram, which proves that  the map $g$ is a fibration.

\item[$(\text{MC}\, 5)$] [\emph{Factorization}]
Let $f : C \to D$ be a morphism of conilpotent dg $\Pac$-coalgebras. The factorization axiom $(\text{MC}\, 5)$ for dg $\Po$-algebras allows us to factor $\Omega_\kappa f$ into 
$$\xymatrix@M=5pt{\Omega_\kappa C \ar[rr]^{\Omega_\kappa f} \ar@{>->}[dr]_{i} &  & \Omega_\kappa D \\
 &   A\ , \ar@{->>}[ur]_p & } $$
 where $i$ is a cofibration and $p$ a fibration and where  one of these two is a quasi-isomorphism. 
So, the morphism $\B_\kappa \Omega_\kappa f$ factors into $\B_\kappa\,   p \circ \B_\kappa \, i$. We consider the following commutative diagram in the category of conilpotent dg $\Pac$-coalgebras. 
$$\xymatrix@M=5pt@R=40pt@C=30pt{\B_\kappa \Omega_\kappa C \ar[rr]^{\B_\kappa \Omega_\kappa f} 
\ar[dr]_{\B_\kappa \, i}& &\B_\kappa \Omega_\kappa D \\
&\B_\kappa  A   \ar[ur]_{\B_\kappa \,  p}& \\
C \ar[rr]^(0.6)f | \hole  \ar[uu]^{\upsilon_\kappa C}  \ar@{..>}[dr]_{\tilde{\i}}  && D \ar[uu]_{\upsilon_\kappa D}\\
&\B_\kappa  A \times_{\B_\kappa \Omega_\kappa D} D \ar[uu] \ar@{}[ruuu] | (0.15)\pullbackbis
  \ar[ur]_{\tilde{p}}& \\}$$
By definition of the pullback, there exists a morphism $\tilde{\i} : C \to \B_\kappa  A \times_{\B_\kappa \Omega_\kappa D} D$, such that $f=\tilde{p}\,  \tilde{\i}$. We shall now prove that the two maps $\tilde{p}$ and  $\tilde{\i}$ are respectively a fibration and a cofibration. 

\noindent
First, the map $\B_\kappa\,  p$ is a fibration by Theorem~\ref{theo:BarCobarFirCofibr}--$(2)$. Since fibrations are stable under base change, the morphism $\tilde{p}$ is also a fibration. 

\noindent
As a cofibration of dg $\PP$-algebras, the map $i : \Ok C \mono A$ is a a retract of a standard cofibration, with isomorphisms on domains, and so is a monomorphism. 
The composite $(\Bk i )(\upsilon_\kappa C)$ is actually equal to the following composite $$ C \xrightarrow{\Delta_C} \Pac(C) \xrightarrow{\Pac(i_{C})} \Pac(A) \ .
$$
Since its first component on $A\cong \I(A)\subset \Pac(A)$ is equal to the restriction $i_C$ on $C$, 
 it is a monomorphism.
We conclude that the morphism $\tilde{\i}$ is a monomorphism by Lemma~\ref{lemm:AcyclicCofibration}, proved in the Appendix~\ref{app:TechLemma}. 

\noindent 
If the map $i$ (resp. $p$) is a quasi-isomorphim, then the map $\Bk \, i$ (resp. $\Bk \, p$) is a weak equivalence by 
Theorem~\ref{theo:BarCobarFirCofibr}--$(2)$. Recall that the unit of adjunction $\upsilon_\kappa$ is a weak equivalence by Theorem~\ref{Thm:KoszulBarCobarRes}--$(2)$. Assuming Lemma~\ref{lemm:AcyclicCofibration}, that is 
 $j : \B_\kappa A \times_{\B_\kappa \Omega_\kappa D} D \stackrel{\sim}{\to} \B_\kappa A$ being a weak equivalence, we conclude that the map $\tilde{\i}$ (resp. $\tilde{p}$) is a weak equivalence by the above axiom $(\text{MC}\, 2)$.

\item[$(\text{MC}\, 4)$] [\emph{Lifting property}] We consider the following commutative diagram in the category of conilpotent dg $\Pac$-colagebras 
$$\xymatrix@R=30pt@C=30pt@M=5pt{E \ar[r]  \ar@{>->}[d]_c & C\  \ar@{->>}[d]^f \\
F \ar[r] \ar@{..>}[ur]^\alpha   & D\ , }$$
where $c$ is a cofibration and where $f$ is a fibration. If moreover the map $c$ is a weak equivalence, then there exists a morphism $\alpha$ such that the two triangles commute, by the definition of the class $\mathfrak{F}$ of fibrations. 

\noindent 
Let us now prove the same lifting property when the map  $f$ is a weak equivalence. Using the aforementioned axiom $(\text{MC}\, 5)$, we factor the map $f$ into $f= \tilde{p}\, \tilde{\i}$, where 
$\tilde{\i}$ is a cofibration and $\tilde{p}$ a fibration. By the axiom $(\text{MC}\, 2)$, both maps 
$\tilde{p}$ and  $\tilde{\i}$
are weak equivalences. By the definition of fibrations, there exists a lifting $r :  \B_\kappa A \times_{\B_\kappa \Omega_\kappa D} D \to C$ in the diagram 
$$\xymatrix@R=30pt@C=30pt@M=5pt{C \ar[r]^{\id_C} \ar@{>->}[d]^\sim_{\tilde{\i}} &  C \ \ar@{->>}[d]^f \\ 
 \B_\kappa A \times_{\B_\kappa \Omega_\kappa D} D  \ar[r]_(0.65){\tilde p} \ar@{..>}[ur]^r  & D\ .} $$
It remains to find a lifting in the diagram 
$$\xymatrix@R=30pt@C=30pt@M=5pt{E \ar[r]  \ar@{>->}[d]_{c} &  \B_\kappa A \times_{\B_\kappa \Omega_\kappa D} D \ar@{->>}[d]^{\tilde p} \\ 
 F   \ar[r] \ar@{..>}[ur]  & D\ ,} $$
which, by the pullback property,  is equivalent to finding a lifting in 
$$\xymatrix@R=30pt@C=30pt@M=5pt{E \ar[r]  \ar@{>->}[d]_{c} &  
\B_\kappa A \times_{\B_\kappa \Omega_\kappa D} D \ar@{->>}[d]^{\tilde p}  \ar[r] \ar@{}[rd] | (0.25)\pullback  & \Bk A \ar[d]^{\Bk \, p} \\ 
 F   \ar[r] \ar@{..>}[urr] | (0.48)\hole   & D \ar[r]_{\upsilon_\kappa D}  &  \Bk \Ok D \ .} $$
To prove that such a lifting exists,  it is enough to consider the following dual diagram under the bar-cobar 
adjunction~\ref{subsec:barcobarAlg}.
$$\xymatrix@R=30pt@C=30pt@M=5pt{\Omega_\kappa E \ar[r]  \ar@{>->}[d]_{\Omega_\kappa  c} & A   \ar@{->>}[d]^p_\sim \\
\Ok F \ar[r] \ar@{..>}[ur]   & \Ok D }$$
Since the cobar construction preserves cofibrations, by Theorem~\ref{theo:BarCobarFirCofibr}--$(1)$, and since the map $p$ 
is an acyclic cofibration, 
we conclude 
by  the lifting axiom $(\text{MC}\, 4)$ in the model category of dg $\Po$-algebras.
\end{itemize}

\end{proof}

\subsection{Fibrant and cofibrant objects}

Since cofibrations are monomorphisms, every conilpotent dg $\Pac$-coalgebra is cofibrant. 
Let us now prove that a conilpotent dg $\Pac$-coalgebra is fibrant if and only if it is isomorphic to a 
quasi-free dg $\Pac$-coalgebra. \smallskip

\begin{proof}(of Theorem~\ref{theo:MCcoalg}--$(2)$)

Let $(C,d_C) \cong (\Pac(A), d_\mu)$ be a conilpotent dg $\Pac$-coalgebra isomorphic to a quasi-free dg $\Pac$-coalgebra. The codifferential $d_\mu$ endows $A$ with a $\PP_\infty$-algebra structure, so $C\cong \widetilde{\B}_\iota A$. 
We consider the unit 
$\upsilon \ : \ A \ {\rightsquigarrow}  \  \Omega_\kappa \widetilde{\B}_\iota A$
of the $(\Ok \widetilde{\B}_\iota, i)$-adjunction of Theorem~\ref{theo:Rectification}. Its first component 
$\upsilon_{(0)}  :  A \cong \I \circ \I \circ A \mono   \PP \circ \Pac \circ A$
 is a monomorphism. We denote by 
 $\rho_{(0)}  :  \PP \circ \Pac \circ A \epi  \I \circ \I \circ A \cong A$
 its right inverse.  We define a map $\rho : \Pac \to \End_A^{\Omega_\kappa \widetilde{\B}_\iota A}$ by the formula of \cite[Theorem~$10.4.1$]{LodayVallette12}. The proof given in loc. cit. shows that the map $\rho$ is an $\infty$-morphism, which is right inverse to $\upsilon$, i.e. $\rho \upsilon = \id_A$.
This allows us to write the conilpotent dg $\Pac$-coalgebra $C$ as a retract of its bar-cobar construction $\Bk \Ok C$:
$$\xymatrix@C=30pt{C \ar[r]^(0.4){\upsilon_\kappa C} \ar@/_1pc/[rr]_{\id}&  \Bk \Ok C \ar[r]^(0.55){\widetilde{\B}_\iota \rho} & C \ .} $$
Since the dg $\PP$-algebra $\Ok C$ is fibrant and since the bar construction $\Bk$ preserves fibrations by 
Theorem~\ref{theo:BarCobarFirCofibr}--$(2)$, then the bar-cobar construction $\Bk \Ok C$ is a fibrant conilpotent dg $\Pac$-coalgebra. We conclude with the general property that fibrant objects are stable under retract. 

In the other way round, let $C$ be a fibrant conilpotent dg $\Pac$-coalgebra. By definition of the fibrations of conilpotent dg $\Pac$-coalgebras, there exists a lifting $r$ in the following commutative diagram 
$$\xymatrix@R=30pt@C=30pt@M=5pt{C \ar@{=}[r]  \ar@{>->}[d]_{\upsilon_\kappa C}^\sim & {\  C\ } \ar@{->>}[d] \\
\Bk \Ok C  \ar[r] \ar@{..>}[ur]^r   &\  0\ , }$$
which makes $C$ into a retract of its bar-cobar construction. Since the map $r$ preserves the respective weight filtrations, 
its first component $\PP(C) \to \mathrm{Prim}\,  C$ on $F_0$ induces the following projection onto $\mathrm{Prim} C$
$$\xi : C \mono \PP(C) \to \mathrm{Prim}\,  C \ . $$ 
Let us now prove that the induced morphism of conilpotent $\Pac$-coalgebras 
$ \Xi : C \to \Pac( \mathrm{Prim}\,  C)$ is an isomorphism.
To that extend, we show, by induction on $n\in \NN$, that the graded morphism 
$$\gr_n \Xi : \gr_n C \xrightarrow{\cong} {\Pac}^{(n)}( \mathrm{Prim}\,  C)\ ,$$
associated to the weight filtration, is an isomorphism. The case $n=0$ is trivially satisfied, since $\gr_0 C= \mathrm{Prim}\,  C$. Suppose now the result true up to $n$ and let us prove it for $n+1$.

We consider the cobar construction $\Omega_\kappa C$ of the $\Pac$-coalgebra $(C, \Delta_C)$ \textit{without} its internal differential, equipped with the filtration induced by the weight filtration of $C$ : 
$$\mathscr{F}_n   \Omega_\kappa C:=\sum_{k\ge 1, \atop l+n_1+\cdots+n_k\leqslant n} F_l  \Po(k)\otimes_{\Sy_k} 
F_{n_1} C \otimes \cdots \otimes F_{n_k} C\ .
  $$
  This filtration is stable under the boundary map $d_2 : \mathscr{F}_n\Omega_\kappa C \to \mathscr{F}_n\Omega_\kappa C$. The associated chain complex 
 $$\gr_n \Omega_\kappa C\cong \sum_{k\ge 1, \atop l+n_1+\cdots+n_k= n} \gr_l  \Po(k)\otimes_{\Sy_k} 
\gr_{n_1} C \otimes \cdots \otimes \gr_{n_k} C$$
is cohomologically graded by the weight $l$ of the operad $\Po$:
\begin{multline*}
\gr_n C \stackrel{\mathrm{d}}{\longrightarrow}
 \sum_{k\ge 1, \atop n_1+\cdots+n_k= n-1}  E(k)\otimes_{\Sy_k} 
\gr_{n_1} C \otimes \cdots \otimes \gr_{n_k} C \stackrel{\mathrm{d}}{\longrightarrow} \\
\sum_{k\ge 1, \atop n_1+\cdots+n_k= n-2}  \gr_2 \Po(k)\otimes_{\Sy_k} 
\gr_{n_1} C \otimes \cdots \otimes \gr_{n_k} C \stackrel{\mathrm{d}}{\longrightarrow} \cdots \stackrel{\mathrm{d}}{\longrightarrow} 
(\gr_n \Po)(\mathrm{Prim}\,  C)\stackrel{\mathrm{d}}{\longrightarrow}0\ .
\end{multline*}
Notice that if one considers the same construction $\gr \, \Omega_\kappa \Pac(V)$ for any the cofree $\Pac$-coalgebra $\Pac(V)$, one gets a (co)chain complex isomorphic to the Koszul complex $q\PP \circ_\kappa q\Pac(V)$ of the quadratic analogous operad $q\PP$ ``with coefficients'' in $V$. Since the Koszul property for the operad $\PP$ includes the Koszul property of the quadratic analgue operad $q\PP$, this later chain complex is acyclic. Recall that it decomposes as a direct sum of sub-chain complexes labelled by the global weight, so that  it is isomorphic and thus quasi-isomorphic to $V$ in weight $0$ and its homology groups are all trivial in higher weights.

The morphism of $\Pac$-coalgebras 
$ \Xi : C \to \Pac( \mathrm{Prim}\,  C)$ induces morphisms of (co)chain complexes, which is $\gr_{n+1} \Omega_\kappa \Xi $ on weight $n+1$:
\begin{eqnarray*}
\xymatrix@C=16pt{0 \ar[r]\ar[d]^\cong& \gr_{n+1} C \ar[r]\ar[d]^{\gr_{n+1} \Xi}&\displaystyle   \sum_{k\ge 1, \atop n_1+\cdots+n_k= n}  E(k)\otimes_{\Sy_k} 
\gr_{n_1} C \otimes \cdots \otimes \gr_{n_k} C \ar[r]\ar[d]^\cong&\cdots\ar[d]^\cong \\
0 \ar[r]& {\Pac}^{(n+1)}( \mathrm{Prim}\,  C) \ar[r]& \displaystyle  \sum_{k\ge 1, \atop n_1+\cdots+n_k= n}  E(k)\otimes_{\Sy_k} 
{\Pac}^{(n_1)}( \mathrm{Prim}\,  C) \otimes \cdots \otimes {\Pac}^{(n_k)}( \mathrm{Prim}\,  C)\ar[r]&\cdots }
\end{eqnarray*}
The top (co)chain complex is acyclic since it can be written as retract of a similar one for a cofree $\Pac$-coalgebra, the bar-cobar resolution $\B_\kappa \Omega_\kappa C$. The bottom one is also acyclic. This allows to conclude that 
the map $\gr_{n+1} \Xi$ is an isomorphism. 
\end{proof}

\subsection{Quillen equivalence}

Now that Theorem~\ref{theo:MCcoalg}--$(1)$ is proved, Theorem~\ref{theo:BarCobarFirCofibr} states that the bar-cobar adjunction 
$$\Bk  \ : \    \textsf{dg} \ \PP\textsf{-alg}  \ \rightleftharpoons\  
\textsf{conil}\ \textsf{dg} \ \Pac\textsf{-coalg}
   \ : \  \Ok \ . $$
forms a Quillen functor.  Let us now prove Point~$(3)$ of Theorem~\ref{theo:MCcoalg}: the bar-cobar adjunction is a Quillen equivalence.

\begin{proof}(of Theorem~\ref{theo:MCcoalg}--$(3)$)$\ $ 

Recall that  any dg $\PP$-algebra is fibrant and that any conilpotent dg $\Pac$-coalgebra is cofibrant, in the respective model category structures considered here. Let $A$ be a dg $\PP$-algebra and let $C$ be a conilpotent dg $\Pac$-coalgebra. We consider two maps 
$$f \ : \  \Ok C \to A  \quad \text{and}  \quad    g \ : \ C \to \Bk A   \ , $$
which are sent to one another under the bar-cobar adjunction. 

If the map $f$ is a quasi-isomorphism of dg $\PP$-algebras, then the map $\Bk f$ is a filtered quasi-isomorphism of conilpotent dg $\Pac$-coalgebras by Proposition~\ref{lemm:QIdonneFilteredQI} and so a weak equivalence by Proposition~\ref{lemm:FiltQIisWE}. Since the map $g$ is equal to the following composite with the unit of adjuction
$$g\ : Ê\ C \xrightarrow{\upsilon_\kappa C} \Bk \Ok C \xrightarrow{\Bk f} \Bk A  \ , $$
then it is a weak equivalence by Theorem~\ref{Thm:KoszulBarCobarRes}--$(2)$.
 
 In the other way round, if the map $g$ is a weak equivalence of conilpotent dg $\Pac$-coalgebras, then the map $\Ok g$ is a quasi-isomorphism of dg $\PP$-algebras by definition. Since the map $f$ is equal to the following composite with the counit of adjuction
$$f\ : Ê\  \Ok C \xrightarrow{\Ok g} \Ok \Bk A  \xrightarrow{\varepsilon_\kappa A}  A \ ,$$
then it is a quasi-isomorphism by Theorem~\ref{Thm:KoszulBarCobarRes}--$(1)$.
\end{proof}

\begin{coro}\label{coro:EquivHo}
The induced adjunction
$$\mathbb{R}\Bk  \ : \   \mathsf{Ho} ( \mathsf{dg} \ \PP\textsf{-}\mathsf{alg} ) \ \rightleftharpoons\  
\mathsf{Ho}(\mathsf{conil}\ \mathsf{dg} \ \Po^{{\acc}}\textsf{-}\mathsf{coalg})
   \ : \  \mathbb{L} \Ok  $$
is an equivalence between the homotopy categories.
\end{coro}


\subsection{Comparison between model category structures on coalgebras}

In order to understand the homotopy theory of conilpotent dg $\Po^{\ac}$-coalgebras with respect to quasi-isomorphisms, one can endow them with a model category structure. 

\begin{theo}[\cite{HessShipley14}]
The category of bounded below dg $\Po^{\ac}$-coalgebras with respectively quasi-iso\-mor\-phi\-sms, degree wise monomorphisms and the induced fibrations forms a model category. 
\end{theo}

\begin{remark}
There is a rich literature on model category structures for coalgebras with respect to quasi-isomorphisms, starting from the original work of Quillen \cite{Quillen67, Quillen69}. Getzler--Goerss treated the case of non-negatively graded but not necessarily conilpotent coassociative coalgebras in \cite{GetzlerGoerss99}. Aubry--Chataur covered the case of conilpotent coalgebras over a quasi-cofree cooperad in \cite{AubryChataur03}.  J.R. Smith worked out in \cite{Smith11} the  case of  coalgebras over operads satisfying a certain condition (Condition~$4.3$ in loc. cit.) and with chain homotopy equivalences.  
\end{remark}

The model category structure with quasi-isomorphisms can be obtained from the present model category structure  by means of Bousfield localization. 

\begin{prop}\label{prop:Bousfield}
The model category structure on conilpotent dg $\PP^{\ac}$-coalgebras with quasi-isomorphisms 
is the left Bousfield localization of the   model category structure of Theorem~\ref{theo:MCcoalg} 
with respect to the class of quasi-isomorphisms.
\end{prop}

\begin{proof}
In this proof, we will denote by $\textsf{coalg}_\textsf{we}$ the category of conilpotent dg $\Po^{\ac}$-coalgebras equipped with the model category structure of Theorem~\ref{theo:MCcoalg} and we will denote 
by $\textsf{coalg}_\textsf{qi}$ the same underlying category but  equipped with the model category structure with quasi-isomorphisms. 

We first prove that the model category $\textsf{coalg}_\textsf{qi}$ is the localization of the model category $\textsf{coalg}_\textsf{we}$ with respect to the class of quasi-isomorphisms. Since the class of weak equivalences sits inside the class of quasi-isomorphisms (Proposition~\ref{prop:weINqi}) and since the cofibrations are the same in both model categories, the identity functor $\id : \textsf{coalg}_\textsf{we} \to \textsf{coalg}_\textsf{qi}$ is a left Quillen functor. (Its right adjoint is the identity too.) We now show that the identity satisfies the universal property of begin a unital object, see Definition~$3.1.1$ of \cite{Hirshhorn03}. Let 
$F  : \textsf{coalg}_\textsf{we} \leftrightharpoons \mathsf{C} : G$
 be a Quillen adjunction such that the total left derived functor $\mathbb{L}F$ sends quasi-isomorphisms into isomorphisms in the homotopy category $\mathrm{Ho}(\mathsf{C})$. Obviously, there is  a unique way to factor this adjunction by the identity adjunction: 
 $\widetilde{F}  : \textsf{coalg}_\textsf{qi} \leftrightharpoons \mathsf{C} : \widetilde{G}$. 
 Since every object in $\textsf{coalg}_\textsf{we}$ is cofibrant, Theorem~$3.1.6$--$(1)$--$(b)$ of \cite{Hirshhorn03} shows that the functor $F$ sends quasi-isomorphisms of coalgebras into weak equivalences in $\mathsf{C}$. 
Therefore, the functor $\widetilde{F}$ is a left Quillen functor. 

We now prove that this localisation of model categories is a Bousfield localization. For this, it is enough to prove that the class of quasi-isomorphisms of coalgebras is equal to the class of local equivalences with respect to quasi-isomorphisms. By definition, the former is  included into the latter. The inclusion in the other way round is provided by 
Theorem~$3.1.6$--$(1)$--$(d)$ of \cite{Hirshhorn03} applied to the identity Quillen adjunction $\id : \textsf{coalg}_\textsf{we} \to \textsf{coalg}_\textsf{qi}$. 
\end{proof}

From the present study and the Bousfield localization, we can obtain a more precise description of the model category with quasi-isomorphisms of \cite{HessShipley14}.

\begin{coro}
In the model category of conilpotent dg $\Po^{\ac}$-coalgebras with quasi-isomorphisms,  the class of acyclic fibrations is the same  as the class of acyclic fibrations in the present model category. Its fibrant objects are the dg $\Pacc$-coalgebras isomorphic to 
quasi-free ones which are local with respect to quasi-isomorphisms. 
\end{coro}

\begin{proof}
The first point is a direct corollary of Proposition~\ref{prop:Bousfield} and Proposition~$3.3.3$--$(1)$--$(b)$ of 
\cite{Hirshhorn03}.
The second point follows from the left properness of the present model category (Theorem~\ref{theo:MCcoalg}--$(2)$), Proposition~\ref{prop:Bousfield} and Proposition~$3.4.1$--$(1)$ of \cite{Hirshhorn03}.
\end{proof}

We refer the reader to Proposition~\ref{prop:InftyqiWE} for a complete description of acyclic fibration between quasi-free $\Po^{\ac}$-coalgebras. For more elaborate results and a full comparison between the possible model category structures on conilpotent dg $\Po^{\ac}$-coalgebras, we refer the reader to the recent preprint of Drummond-Cole--Hirsh  \cite{DrummondColeHirsh14}.


\section{Homotopy theory of  infinity-morphisms}
The purpose of this section is to apply the previous model category structure on conilpotent dg $\PP^{\ac}$-coalgebras to get general results about $\infty$-morphisms.
For instance, the model category structure provides us automatically with a good notion of homotopy equivalence between morphisms of fibrant-cofibrant objects, that is a homotopy equivalence between $\infty$-morphisms of $\PP_\infty$-algebras. 
In this section, we realize this homotopy equivalence with a functorial cylinder object. 
We also show that this new simple homotopy equivalence is equivalent to ``all'' the equivalence relations that have been considered so far on $\infty$-morphisms in the literature. 
Then, we state and prove one of the main results of this paper: the homotopy category of $\PP$-algebras is equivalent to the category of $\PP$-algebras with $\infty$-morphisms up to homotopy equivalence. 
Finally, we explain how  the present model category does not only encaptures this first level homotopical data, but 
all the higher homotopical properties of $\PP_\infty$-algebras. This is achieved by upgrading the category of $\PP_\infty$-algebras with $\infty$-morphism into an $\infty$-category. 


\subsection{Functorial cylinder objects}
To define functorial cylinder objets in the category of dg $\PP^{\ac}$-coalgebras, we consider two algebraic models for the interval, the first one in the category of dg coassociative coalgebras and the second one in the category of dg commutative algebras. 

\begin{defi}[Coassociative model for the interval] We consider the cellular chain complex of the interval: 
$$I:=\KK \mathfrak{0}\oplus \KK \mathfrak{1} \oplus \KK \mathfrak{i}, \ \text{with}\ |\mathfrak{0}|=|\mathfrak{1}|=0, \ |\mathfrak{i}|=1, \  \text{and} \ \  d(\mathfrak{0})=d(\mathfrak{1})=0, \  d(\mathfrak{i})=\mathfrak{1}-\mathfrak{0} \ . $$
It is equipped with a dg coassociative coalgebra structure by 
$$\Delta(\mathfrak{0})=\mathfrak{0}\otimes \mathfrak{0},\  \Delta(\mathfrak{1})=\mathfrak{1}\otimes \mathfrak{1}, \ \text{and}\ \Delta(\mathfrak{i})=\mathfrak{0}\otimes \mathfrak{i} + \mathfrak{i} \otimes \mathfrak{1}\ . $$
\end{defi}

When $\PP$ is a nonsymmetric operad, the tensor product of any conilpotent dg $\PP^{\ac}$-coalgebra $(C, d_C, \Delta_C)$ with $I$ provides us with a functorial conilpotent dg $\PP^{\ac}$-coalgebra. The arity $n$ component of its structure map is given by 
$$ C\otimes I \xrightarrow{\Delta_C(n)\otimes \Delta^{n-1}}  \PP^{\ac}(n)\otimes C^{\otimes n}\otimes I^{\otimes n}
\xrightarrow{\cong} \PP^{\ac}(n)\otimes ( C\otimes I)^{\otimes n}\ .
$$ 

\begin{prop}
Let $\PP$ be a nonsymmetric Koszul operad and let $C$ be a conilpotent dg  $\Pacc$-coalgebra. The $\Pacc$-coalgebra  $C\otimes I $ provides us with a functorial good cylinder object 
$$ 
\xymatrix@C=25pt{
C \oplus C \ar@{>->}[r] &  C\otimes I \ar[r]^(0.58)\sim &  C }
$$
in the model category of conilpotent dg $\Pacc$-coalgebras of Theorem~\ref{theo:MCcoalg}. 
\end{prop}

\begin{proof}
 One can notice that the weight filtration satisfies $F_n(C\otimes I)\cong F_n C \otimes I$. Since the $\Pacc$-coalgebra $C$ is conilpotent, then so is $C\otimes I$. The left-hand map is the embedding $c+c'\mapsto  c\otimes\mathfrak{0} + 
 c'\otimes\mathfrak{1}$, hence it is a cofibration. 
The right-hand map is equal to 
$c\otimes\mathfrak{0} \mapsto c$, $c\otimes\mathfrak{1} \mapsto c$, and $c\otimes\mathfrak{i} \mapsto 0$. To prove that it is a weak-equivalence, we show that it is a filtered quasi-isomorphism.  The graded part of the above morphism of dg $\PP^{\ac}$-coalgebras is equal to 
$$ F_n C / F_{n-1} C \otimes  I \to  F_n C / F_{n-1} C\ ,$$
with the same kind of formula. So, this is a quasi-isomorphism and we conclude with Proposition~\ref{lemm:FiltQIisWE}. 
\end{proof}

In the symmetric case, the situation is more involved since it is difficult to find a suitable model for the interval in the category of dg cocommutative coalgebras equipped with two different group-like elements. 
Instead, we proceed as follows. 

\begin{defi}[Commutative model for the interval] \leavevmode
Let $\Lambda(t, dt):= \KK[t] \oplus \KK[t]dt$ be the dg commutative algebra made up of the polynomial differential forms on the interval. The  degree is defined by 
$|t|=0$, $|dt|=-1$ and the differential is the unique derivation extending $\dd(t)=dt$
 and $\dd(dt)=0$. 
This \emph{dg commutative algebra model for the interval} 
$(\Lambda(t, dt), \dd)$ 
is called the \textit{Sullivan algebra}. 
\end{defi}

Let $\PP$ be a Koszul operad and let $(A, \mu)$ be a $\PP_\infty$-algebra. The tensor product $A\otimes \Lambda(t, dt)$ inherits a natural $\PP_\infty$-algebra, given by 
$$\widetilde{\mu} \ : \ \PP^{\ac} \cong \PP^{\ac} \otimes Com
\xrightarrow{\mu\otimes \nu} {\End}_A \otimes {\End}_{\Lambda(t,dt)}\cong 
{\End}_{A\otimes\Lambda(t,dt)}
\ ,$$
where $Com$ denotes the operad of commutative algebras, whose arity-wise components are  one-di\-men\-sion\-al, and where $\nu$ denotes the commutative algebra structure on $\Lambda(t, dt)$. 
We consider the cellular chain complex of the interval 
$I^*$, which is isomorphic to the sub-complex of $\Lambda(t, dt)$ made up of 
$J:=\KK 1 \oplus \KK t \oplus \KK dt$ under the identification $\mathfrak{0}^*+\mathfrak{1}^*\leftrightarrow 1$, 
$\mathfrak{1}^*\leftrightarrow t$, and $\mathfrak{i}^*\leftrightarrow dt$. This latter chain complex is a deformation retract of the polynomial differential forms on the interval; a particularly elegant contraction was given by J. Dupont  in his proof of the de Rham theorem \cite{Dupont76}, see also \cite{ChengGetzler08, Getzler09}. 

\begin{defi}[Dupont's contraction]
The Dupont's contraction amounts to the following deformation retract: 
\begin{eqnarray*}
&\xymatrix{     *{ \quad \ \  \quad\quad\ \  \  (\Lambda(t, dt), \dd)\ } \ar@(dl,ul)[]^{\textrm{h}}\ \ar@<0.5ex>[r]^(0.6){\pp} & *{\
(J,\dd)\quad \ \  \ \quad }  \ar@<0.5ex>[l]^(0.4){\ii}}&\\
& \hh(t^kdt):=\frac{t^{k+1}-t}{k+1}, \ \hh(t^k):=0 \quad \text{and} \quad \pp(t^k dt):=\frac{1}{k+1}dt, \ \pp(1):=1, \ \pp(t^k)=t \ \  \text{for} \ \ k\ge 1
\ .
\end{eqnarray*}
\end{defi} 
We now consider the  contraction $\id_A \otimes h$ on $A\otimes \Lambda(t,dt)$ and then the induced $\PP_\infty$-algebra structure $\hat{\mu}$ on $A\otimes J$ obtained by applying the Homotopy Transfer Theorem \cite[Theorem~$ 10.3.1$]{LodayVallette12}, see also \cite[Section~$8$]{DotsenkoShadrinVallette15} for more insight.  
This  transferred $\PP_\infty$-algebra structure satisfies the following properties. 

\begin{lemm}\label{lem:HTTProperty1}
The $\PP_\infty$-algebra structure $\hat{\mu}$ on $A\otimes \Lambda(t,dt)$ satisfies 
\begin{eqnarray*}
\hat{\mu}(p; a_1\otimes b_1, \ldots, a_n\otimes b_n) &=& {\mu}(p; a_1, \ldots, a_n)\otimes \pp(b_1\cdots b_n) \ ,
\end{eqnarray*}
for any $p\in \PP^{\ac}(n)$,  $a_1, \ldots, a_n \in A$ when $b_1, \ldots, b_n =1\ \text{or} \ t$ and 
\begin{eqnarray*}
\hat{\mu}(p; a_1\otimes b_1, \ldots, a_i\otimes dt, \ldots, a_n\otimes 1) \in A \otimes \KK dt \ , 
\end{eqnarray*}
for any $p\in \PP^{\ac}(n)$, $a_1, \ldots, a_n \in A$, and $b_1, \ldots, b_{i-1}, b_{i+1}, \ldots, b_n\in J$. 
\end{lemm}

\begin{proof}[Proof of Lemma~\ref{lem:HTTProperty1}]
The formula for the Homotopy Transfer Theorem given in \cite[Theorem~$10.3.3$]{LodayVallette12} is a sum of terms where one has always to apply the homotopy $\hh$ to some of the $b_i$'s, except for the term 
${\mu}(p; a_1, \ldots, a_n)\otimes \pp(b_1\cdots b_n)$. Since $\hh$ applied to $t^k$ gives $0$, when all the $b_i$'s are equal to $1$ or $t$, only remains the last term. 

When one $b_i$ is equal to $dt$, the last term is of the form ${\mu}(p; a_1, \ldots, a_n)\otimes P(t)dt$, with $P(t)\in \KK[t]$. Each other term involves applying at least one homotopy $\hh$ above the root vertex. Therefore the upshot is of the form $a \otimes \pp\big(\sum_{l=1}^L P_l(t)(t^{k_l+1} - t)\big)$, with $a\in A$ and $P_l(t)\in \KK[t]$. Since $\pp\big(t^m(t^{k+1}-t)\big)=0$, all these terms vanish and the second formula is proved. 
\end{proof}

Finally, this produces the required cylinder for quasi-free dg $\Pac$-coalgebras.

\begin{prop}\label{prop:Cylinder}
Let $\PP$ be a Koszul operad and let $(\PP^{\ac}(A), d_\mu)$ be a quasi-free dg $\Pac$-coalgebras. The $\Pacc$-coalgebra  $(\PP^{\ac} (A\otimes J), d_{\bar{\mu}})$ provides us with a functorial good cylinder object 
$$ 
\xymatrix@C=25pt{
\PP^{\ac}(A) \oplus \PP^{\ac}(A) \ar@{>->}[r] & \PP^{\ac} (A\otimes J) \ar[r]^(0.58)\sim &  \PP^{\ac}(A)}
$$
in the model category of conilpotent dg $\Pacc$-coalgebras of Theorem~\ref{theo:MCcoalg}.  
\end{prop}

\begin{proof}
For the left-hand map, we consider the embedding 
$i_0+i_1\ : a+b \mapsto 
a\otimes 1 + b\otimes t$, which extends to the following unique morphism of $\PP^{\ac}$-coalgebras:
$$\PP^{\ac}(i_0)+\PP^{\ac}(i_1)\ : \ p(a_1, \ldots, a_n)+q(b_1, \ldots, b_n) \mapsto 
p(a_1\otimes 1, \ldots, a_n\otimes 1)+q(b_1\otimes t, \ldots, b_n\otimes t)\ , 
 $$
for $p, q \in \PP^{\ac}(n)$ and for $a_1, \ldots, a_n, b_1, \ldots, b_n \in A$. 
It commutes with the respective differentials since the following diagram, and the other similar one with $i_1$, are commutative 
$$\xymatrix@C=52pt{\PP^{\ac}(A) \ar[rrr]^{\PP^{\ac}(i_0)} \ar[dd]^{d_\mu} \ar[dr]^{\Delta_{(1)}(A)}& & & \PP^{\ac}(A\otimes \KK 1)\ar[dd]^{d_{\hat{\mu}}}\ar[dl]_{\Delta_{(1)}(A\otimes \KK 1)\ \ \ }\\
& \PP^{\ac}\big(\PP^{\ac}(A); A\big) \ar[dl]^(0.4){\PP^{\ac}(\mu_A; A)} \ar[r]^(0.4){\PP^{\ac}\big(\PP^{\ac}(i_0); i_0\big)}& \PP^{\ac}\big(\PP^{\ac}(A\otimes \KK 1); A\otimes \KK 1\big) \ar[dr]_(0.4){\PP^{\ac}(\hat{\mu}_{A\otimes \KK 1}; A\otimes \KK 1)\ \ }& \\
\PP^{\ac}(A)\ar[rrr]^{\PP^{\ac}(i_0)} & & & \PP^{\ac}(A\otimes \KK 1)} $$
by Lemma~\ref{lem:HTTProperty1}, where $\mu_A$ (respectively $\hat{\mu}_{A \otimes \KK 1}$) denotes the map $\PP^{\ac}(A) \to A$ induced from $\mu : \PP^{\ac} \to \End_A$. 
It is clearly a degreewise monomorphism, and so a cofibration. 

The second map $\PP^{\ac} (A\otimes J)\to \PP^{\ac}(A)$ is the unique morphism $\PP^{\ac}(j)$ of $\PP^{\ac}$-coalgebras which extends the map $j : A\otimes J \to A$, defined by  $j(a\otimes 1)=j(a\otimes t)=a$ and $j(a\otimes dt)=0$. It 
 commutes with the respective differentials  since the following diagram is commutative 
$$\xymatrix@C=52pt{\PP^{\ac}(A\otimes J) \ar[rrr]^{\PP^{\ac}( j)} \ar[dd]^{d_{\hat{\mu}}} \ar[dr]^{\Delta_{(1)}(A\otimes J)}& & & \PP^{\ac}(A)\ar[dd]^{d_{{\mu}}}\ar[dl]_{\Delta_{(1)}(A)\ \ \ }\\
& \PP^{\ac}\big(\PP^{\ac}(A\otimes J); A\otimes J\big) \ar[dl]^(0.4){\ \ \PP^{\ac}(\hat{\mu}_{A\otimes J}; A\otimes J)} \ar[r]^(0.57){\PP^{\ac}\big(\PP^{\ac}(j);j\big)}& \PP^{\ac}\big(\PP^{\ac}(A); A\big) \ar[dr]_(0.4){\PP^{\ac}({\mu}_{A}; A)\ \ }& \\
\PP^{\ac}(A\otimes J)\ar[rrr]^{\PP^{\ac}(j)} & & & \PP^{\ac}(A)} $$
by Lemma~\ref{lem:HTTProperty1}.
Notice the map  $j : A\otimes J \to A$ is a quasi-isomorphism. We consider the canonical weight filtrations on the quasi-free dg $\PP^{\ac}$-coalgebras $\PP^{\ac}(A\otimes J)$ and $\PP^{\ac}(A)$. They induce a filtered quasi-isomorphism and thus a weak equivalence by Proposition~\ref{lemm:FiltQIisWE}. 

The data  a morphism $F: (\PP^{\ac}(A), d_\mu) \to (\PP^{\ac}(B), d_\omega)$ between two quasi-free dg $\Pac$-coalgebras is equivalent to the data of an $\infty$-morphism $f: (A, \mu) \rightsquigarrow (B, \omega)$ between the two associated $\PP_\infty$-algebras. The Homotopy Transfer Theorem \cite[Theorem~$10.3.3$]{LodayVallette12} produces the following natural $\infty$-morphisms
$$\xymatrix@C=50pt{(A\otimes \Lambda(t,dt), \widetilde{\mu})  \ar@{~>}[d]^{f\otimes \nu}& \ar@{~>}[l]_(0.45){(A\otimes \ii)_\infty} (A\otimes J, \hat{\mu}) \ar@{~>}[d]\\
(B\otimes \Lambda(t,dt), \widetilde{\omega})  \ar@{~>}[r]^(0.55){(B\otimes \pp)_\infty}& (B\otimes J, \hat{\omega})}\ ,$$
 whose composition gives the functoriality of the present cylinder object. 
\end{proof}

\subsection{Homotopy equivalence of $\infty$-morphisms}
The ultimate goal of these two sections, is to provide the present theory with a suitable explicit notion of homotopy equivalence of $\infty$-morphisms, which allows us to obtain a simple description of the homotopy category of $\PP$-algebras, for instance. 

\begin{defi}[Homotopy relation]
Two morphisms $f,g : C \to D$ of conilpotent dg  $\Pacc$-coalgebras are \emph{homotopic} if there exists 
a morphism $h$ of conilpotent dg  $\Pacc$-coalgebras fulling the commutative diagram 
$$
\xymatrix@C=25pt{
C  \ar[r]^(0.4){j_0}\ar[dr]_{f}& \mathrm{Cyl(C)} \ar@{..>}[d]^{h} & \ar[l]_(0.4){j_1} \ar[dl]^{g}C \\
& \ D\ , &
}
 $$
 where $\mathrm{Cyl}(C)$ is a cylinder for $C$. 
\end{defi}

\begin{prop}
This homotopy relation of morphisms of conilpotent dg  $\Pacc$-coalgebras is an equivalence relation. 
\end{prop}

\begin{proof}
This a direct consequence of the general theory of model categories \cite[Proposition~$1.2.5$]{Hovey99} since every object of the present category is cofibrant by Theorem~\ref{theo:MCcoalg}--$(2)$.
\end{proof}

This homotopy relation restricts naturally to maps between fibrant-cofibrant objects, that is to $\infty$-morphisms of $\PP_\infty$-algebras. In this case, we can use the small cylinder given in Proposition~\ref{prop:Cylinder} for instance. 
In the recent paper \cite{DotsenkoPoncin12}, V. Dotsenko and N. Poncin study several equivalence relations for $\infty$-morphisms and they prove that they are all equivalent. The next proposition shows that these equivalence relations are actually homotopy equivalences in the sense of the present model category of conilpotent dg  $\Pac$-coalgebras.

\begin{prop}
All the equivalence relations for $\infty$-morphisms of \cite{DotsenkoPoncin12} are equivalent to the above homotopy relation. 
\end{prop}

\begin{proof}
It is enough to  prove that 
$$ 
\xymatrix@C=25pt{
\PP^{\ac}(A)  \ar@{->}[r]^(0.35)\sim & \PP^{\ac} \big(A\otimes \Lambda(t, dt)\big) \ar@{->>}[r] &  \PP^{\ac}(A\oplus A)}
$$
is a good path object in the model category of conilpotent dg $\Pac$-coalgebras of Theorem~\ref{theo:MCcoalg}.  Then, by the general theory of model categories, the associated right homotopy equivalence will be an equivalence relation, equivalent to the above left homotopy relation defined by the above good cylinder object. Finally, one can notice that this new equivalence is nothing but the one called concordance in \cite[Definition~3]{DotsenkoPoncin12}, which is proved in loc. cit. to be equivalent to the other ones. 

This statement is proved in the same way as Proposition~\ref{prop:Cylinder}, so we will only give the various arguments and constructions. 
We first notice that the right-hand term is the categorical product of the dg $\Pac$-coalgebra $(\Pac(A), d_\mu)$ with itself; the details can be found in the proof of Axiom $(\text{MC}\, 1')$ of Theorem~\ref{theo:MCcoalgPinfty}--$(1)$ in the sequel. 
The first map is defined by the unique morphism of $\Pac$-coalgebras extending 
$a \mapsto a\otimes 1$. It is straightforward to check that it commutes with the differentials. Since it is the extension of a quasi-isomorphism $A\qi A\otimes \Lambda(t,dt)$, it is a filtered quasi-isomorphism and therefore a weak equivalence of conilpotent dg $\Pac$-coalgebras. The second map is defined by the unique morphism of $\Pac$-coalgebras extending 
$a\otimes (P(t)+Q(t)dt) \mapsto a\otimes P(0) + a\otimes P(1)$. It is again straightforward to check that it commutes with the differentials. 
To prove that it forms a fibration, we use the characterization of fibrations between quasi-free dg $\Pac$-coalgebras given in Proposition~\ref{prop:InftyqiWE} (and whose proof does not depend on the present result). Since the map $A\otimes \Lambda(t,dt) \epi A \oplus A$ is a degreewise epimorphism, the map $\PP^{\ac} \big(A\otimes \Lambda(t, dt)\big) \epi \PP^{\ac}(A\oplus A)$ is a fibration. 
Finally, we check that the composite of these two maps is equal to the product of the identity with itself, which concludes the proof. 
\end{proof}

The following result refines Theorem~\ref{InverseInftyQI}: it gives a finer control of the ``inverse'' $\infty$-quasi-isomorphism. 

\begin{theo}\label{theo:InftyQiInv}
Any $\infty$-quasi-isomorphism admits a homotopy inverse. 
\end{theo}

\begin{proof}
By Proposition~$11.4.7$ of \cite{LodayVallette12}, any $\infty$-quasi-isomorphism $f : A \stackrel{\sim}{\rightsquigarrow} A'$ induces a weak equivalence 
$$\widetilde{\B}_\iota f \, : \, \widetilde{\B}_\iota A \qi \widetilde{\B}_\iota A'$$
between fibrant-cofibrant conilpotent dg $\Pac$-coalgebras. By the general model category arguments, this latter one admits a homotopy inverse, which translated back on the level of $\PP_\infty$-algebras gives the result. 
\end{proof}

The other general consequence of the above mentioned theory is the following description of the homotopy category of dg $\PP$-algebras, defined as a localized category, as the  category of $\PP$-algebras with $\infty$-morphism up to homotopy equivalence.

\begin{theo}\label{theo:EquivalenceHoMAIN}
The following categories are equivalent
$$ 
\mathsf{Ho}(\mathsf{dg}\  \PP\textsf{-}\mathsf{alg}) \ \simeq\  
\mathsf{Ho}( \infty\textsf{-}\PP_\infty\textsf{-}\mathsf{alg})\ \simeq\ 
 \infty\textsf{-}\PP_\infty \textsf{-}\mathsf{alg}/\sim_{h} \ \simeq\ 
 \infty\textsf{-}\PP\textsf{-}\mathsf{alg}/\sim_{h} \ .$$ 
\end{theo}

\begin{proof}
By Corollary~\ref{coro:EquivHo} of Theorem~\ref{theo:MCcoalg}--$(3)$, the total derived functors  
$$\mathbb{R}\Bk  \ : \   \mathsf{Ho} ( \mathsf{dg} \ \PP\textsf{-}\mathsf{alg} ) \ \stackrel{\cong}{\rightleftharpoons}\  
\mathsf{Ho}(\mathsf{conil}\ \mathsf{dg} \ \Pac\textsf{-}\mathsf{coalg})
   \ : \  \mathbb{L} \Ok  $$
form an equivalence of categories. By the general arguments of model categories, the right-hand category is equivalent to the category of fibrant-cofibrant conilpotent dg $\Pac$-coalgebras modulo the homotopy relation. So by Theorem~\ref{theo:MCcoalg}--$(2)$, we get the following equivalence of categories
$$\mathsf{Ho}(\mathsf{conil}\ \mathsf{dg} \ \Pac\textsf{-}\mathsf{coalg})\stackrel{\cong}{\rightleftharpoons}
  \textsf{quasi}\textsf{-}\textsf{free}\  \PP^{\ac}\textsf{-}\textsf{coalg}/\sim_h \ .$$
  We then use the equivalence between the category of quasi-free $\Pac$-coalgebras and the category of $\PP_\infty$-algebras with their $\infty$-morphisms. Finally, we conclude with Theorem~\ref{theo:Rectification}, which shows that any $\PP_\infty$-algebra can be rectified into a dg $\PP$-algebra:
$$\infty\textsf{-}\PP_\infty\textsf{-}\mathsf{alg}/\sim_{h}\
\stackrel{\cong}{\rightleftharpoons} \ \infty\textsf{-}\PP\textsf{-}\mathsf{alg}/\sim_{h} \ .$$ 
\end{proof}

Using the explicit homotopy relation defined by the cylinder object associated to $I$, this recovers the classical homotopy relation of $A_\infty$-morphisms and the description of the homotopy category for \emph{unbounded} dg associative algebras, see  \cite{Munkholm78, LefevreHasegawa03}. 


\subsection{An $\infty$-category enrichment of homotopy algebras}
The previous result deals with the homotopy category of $\PP_\infty$-algebras, which is only the first  homotopical level of information. At the present stage of the theory, we have objects (the $\PP_\infty$-algebras), $1$-morphisms (the $\infty$-morphisms) and $2$-morphisms (the homotopy relation). However, one can go further, thanks to the model category structure established in the previous section, and prove that the category of $\PP_\infty$-algebras actually extends to an $\infty$-category.

\begin{theo}\label{thm:inftycat}
The category $ \infty\textsf{-}\PP_\infty\textsf{-}\mathsf{alg}$ of $\PP_\infty$-algebras 
with $\infty$-morphisms extends to a simplicial category giving the same underlying homotopy category. 
\end{theo}

\begin{proof}
This is a direct application of the simplicial localization methods of Dwyer--Kan \cite{DwyerKan80}.
\end{proof}

During the preparation of this paper, Dolgushev--Hoffnung--Rogers \cite{DolgushevHoffnungRogers14} used the integration theory of $L_\infty$-algebras \cite{Getzler09} to endow the category of $\PP_\infty$-algebra with another model of $\infty$-category. The comparison between these two $\infty$-category structures 
is an interesting subject for further studies. 


\section{Homotopy algebras}

The model category structure  for algebras over an operad of Theorem~\ref{theo:Hinich97} applies as well to the category $\PP_\infty$\textsf{-alg} of $\PP_\infty$-algebras with their \emph{strict} morphisms. But if we consider the category $\infty$\textsf{-}$\PP_\infty$\textsf{-alg} of $\PP_\infty$-algebras with their $\infty$-morphisms, then it cannot admit a model category structure strictly speaking since it lacks some colimits like  coproducts.

With the abovementioned isomorphism between the category of quasi-free $\Pac$-coalgebras and the category of $\PP_\infty$-algebras with their $\infty$-morphisms, Theorem~\ref{theo:MCcoalg} shows that the category $\infty$\textsf{-}$\PP_\infty$\textsf{-alg} is endowed with a \emph{fibrant objects category structure}, see \cite{Brown73} for the definition. Such a notion is defined by  two classes of maps: the weak equivalences and the fibrations. Notice that  the fibrations of $\PP^{\ac}$-coalgebras has not been made explicit so far.

In this section, we refine this result: we provide the category $\infty$\textsf{-}$\PP_\infty$\textsf{-alg} of homotopy $\PP$-algebras and their $\infty$-morphisms with almost a {model category structure}, only the first axiom  on limits and colimits is not completely fulfilled. 
As a consequence, this will allow us to describe the fibrations between quasi-free $\PP^{\ac}$-coalgebras.

\subsection{Almost a model category}

\begin{defi}
In the category of homotopy $\PP$-algebras with their $\infty$-morphisms, we consider the following three classes of morphisms. 

\begin{itemize}
\item[$\diamond$] The class $\mathfrak W$ of \emph{weak equivalences} is given by the $\infty$-quasi-isomorphisms  $f : A \stackrel{\sim}{\rightsquigarrow} A'$, i.e. the $\infty$-morphisms whose first component $f_{(0)} : A \qi A'$ is a quasi-isomorphism;  

\item[$\diamond$] the class $\mathfrak C$ of \emph{cofibrations} is given by the $\infty$-monomorphisms  
$f :A \rightsquigarrow A'$, 
i.e. the $\infty$-morphisms whose first component $f_{(0)} : A  \mono A'$ is a monomorphism;  

\item[$\diamond$] the class $\mathfrak F$ of \emph{fibrations} 
is given by $\infty$-epimorphisms  $f : A \rightsquigarrow A'$, i.e. the $\infty$-morphisms whose first component $f_{(0)}  : A \epi A'$ is a epimorphism;  

\end{itemize}
\end{defi}

\begin{theo}\label{theo:MCcoalgPinfty}\leavevmode
\begin{enumerate}

\item The category $\infty$\textsf{-}$\PP_\infty$\emph{\textsf{-alg}} of $\PP_\infty$-algebras with their $\infty$-morphisms, endowed with the three classes of maps $\mathfrak W$, $\mathfrak C$, and $\mathfrak F$, satisfies the axioms $(\mathrm{MC}\, 2)$--$(\mathrm{MC}\, 5)$ of model categories and the following axiom.

\smallskip
\noindent
$(\mathrm{MC}\, 1')\quad $ This category admits finite products and pullbacks of fibrations.
\smallskip

\item Every $\PP_\infty$-algebra is fibrant and cofibrant.
\end{enumerate}
\end{theo}

Recall that a category admits finite colimits if and only if it admits finite coproducts and coequalizers.
The present category lacks coproducts. It is enough to consider the two 
dimension $1$ trivial $A_\infty$-algebras $\bar{T}^c(sx)$ and $\bar{T}^c(sy)$, viewed as quasi-free coassocaitive coalgebras, and to see that they do not admit coproducts in that category. 
The situation about equalizers and coequalizers is more subtle and requires further studies.

\subsection{Properties of  $\infty$-morphisms}
The proof of Theorem~\ref{theo:MCcoalgPinfty} relies on the algebraic properties of $\infty$-morphism given in this section and on the obstruction theory developped in Appendix~\ref{App:Obstruction}.\\

Recall from \cite{LodayVallette12} that the dg module $\Hom_\Sy(\PP^{\ac}, {\End}_A)$ 
is endowed with a preLie product defined by
 $$f\star g:=   {\gamma}_{\PP} \circ  \big(f\otimes g \big) \circ    \Delta_{(1)}\ ,$$
where $\Delta_{(1)} : \PP^{\ac}\to \TTT(\PP^{\ac})^{(2)}$ is the partial decomposition map of the cooperad $\PP^{\ac}$. This preLie product induces the Lie bracket of Section~\ref{subsec:OHA} by antisymetrization $[f,g]:= f\star g - (-1)^{|f||g|}g\star f$. So the Maurer--Cartan equation encoding $\PP_\infty$-algebra structures  is equivalently written 
$$\partial (\alpha) + \alpha \star \alpha=0\ . $$

We consider the dg $\Sy$-module  $\End^A_B$ defined by 
$${\End}^A_B:=\big(\lbrace \Hom(A^{\t n}, B)   \rbrace_{n\in \NN}, \partial^A_B\big)\ . $$ 
Let $\mu \in \Hom_\Sy(\PP^{\ac}, {\End}_A)$, $\nu \in \Hom_\Sy(\PP^{\ac}, {\End}_B)$ and $f \in \Hom_\Sy(\PP^{\ac}, {\End}^A_B)$. 
We consider the following two operations:
\begin{eqnarray*}
f\ast \mu &:=& \PP^{\ac} \xrightarrow{\DD_{(1)}} \PP^{\ac} \circ_{(1)} {\PP}^{\ac}
\xrightarrow{f\circ_{(1)}\mu} {\End}_B^A\circ_{(1)}{\End}_A \to {\End}^A_B \ , \\
\nu\circledast f &:=& \PP^{\ac} \xrightarrow{\DD} {\PP}^{\ac} \circ {\PP}^{\ac}
\xrightarrow{\nu \circ f} {\End}_B\circ {\End}_B^A \to {\End}_B^A\ , 
\end{eqnarray*}
where the right-hand maps is the usual composite of functions. 

\begin{theo}[Theorem~$10.2.3$  of \cite{LodayVallette12}]\label{theo:EquivInftyMorph}
Let $(A, d_A, \mu)$ and $(B, d_B, \nu )$ be two $\PP_\infty$-algebras. 
An $\infty$-morphism $F : A \rightsquigarrow B $  of $\PP_\infty$-algebras  is equivalent to a morphism of dg $\Sy$-modules $f : {\PP}^{\ac} \to \End^A_B$
satisfying 
\begin{eqnarray}\label{eqn:inftyMorphEquiv}
\partial(f) =f * \mu - \nu \circledast f 
\end{eqnarray}
in $\Hom_\Sy(\PP^{\ac},\End_B^A)$.
\end{theo}
Using this equivalent definition of $\infty$-morphisms, the composite of  $f : A \rightsquigarrow B$ with $g : B\rightsquigarrow C$ is given by 
$$g \circledcirc f :={\PP}^{\ac} \xrightarrow{{\DD}} {\PP}^{\ac} \circ {\PP}^{\ac} \xrightarrow{g \circ f} {\End}^B_C \circ {\End}^A_B \to {\End}^A_C\ .$$
Notice that the product $\circledcirc$ is associative and left linear. 

For any element $\mu \in \Hom_\Sy(\PP^{\ac}, {\End}_A)$ and any element $f \in \Hom_\Sy(\PP^{\ac}, {\End}^A_B)$,  we denote by 
$$\mu_{(n)} \in \Hom_\Sy({\PP^{\ac}}^{(n)}, {\End}_A) \ \text{and by} \ f_{(n)}\in \Hom_\Sy({\PP^{\ac}}^{(n)}, {\End}^A_B) $$ the respective restrictions to 
the weight $n$ part ${\PP^{\ac}}^{(n)}$ of the cooperad $\PP^{\ac}$.
Recall that an $\infty$-morphism $f : A \rightsquigarrow B$ of $\PP_\infty$-algebras is a (strict) morphism of $\PP_\infty$-algebras if and only if its higher components $f_{(n)}=0$ vanish for $n\ge 1$.

\begin{prop}\label{prop:strification}$ \ $

\begin{enumerate}

\item Let $f : A \rightsquigarrow B$ be an $\infty$-monomorphism. There exists a $\PP_\infty$-algebra $C$ together with an $\infty$-isomorphism $ g : B \rightsquigarrow C$ such that the composite $gf$ is a (strict) morphism of $\PP_\infty$-algebras. 

\item  Let $g :  B\rightsquigarrow C$ be an $\infty$-epimorphism. There exists a $\PP_\infty$-algebra $A$ together with an $\infty$-isomorphism $ f : A \rightsquigarrow B$ such that the composite $gf$ is a (strict) morphism of $\PP_\infty$-algebras.
\end{enumerate}
\end{prop}

\begin{proof}$ \ $

\begin{enumerate}

\item Since the map $f_{(0)} : A\mono B$ is a monomorphism of graded modules, it admits a retraction $r : B \epi A$, such that $r f_{(0)}= \id_A$. We define a series of linear maps $g_{(n)} : {\Pac}^{(n)} \to \End_B$ by induction as follows. Let $g_{(0)}$ be equal to ${\Pac}^{(0)}=\I \mapsto \id_B$. Suppose the maps $g_{(k)}$ constructed up to $k=n-1$, we define the map $g_{(n)}$ by the formula 
$$g_{(n)}:=-
\sum_{k=0}^{n-1} g_{(k)} \circledcirc (r^*f)
 \ , $$
where the map $r^*f$ is equal to the composite 
$$r^*f : \Pac(n) \xrightarrow{f} \Hom(A^{\otimes n}, B) \xrightarrow{(r^{\otimes n})^*} \Hom(B^{\otimes n}, B)\ . $$
So for $n\ge 1$, the weight $n$ part of the composite $gf$ is equal to 
\begin{eqnarray*}
(gf)_{(n)}=\sum_{k=0}^n g_{(k)}\circledcirc f = \sum_{k=0}^{n-1} g_{(k)}\circledcirc f + g_{(n)} \circledcirc f_{(0)}
= \sum_{k=0}^{n-1} \big( g_{(k)}\circledcirc f       -g_{(k)}\circledcirc (\underbrace{r^*f\circledcirc f_{(0)}}_{=f\circledcirc (rf_{(0)})=f}) \big)  =0\ .
\end{eqnarray*}
Since the image of $g_{(0)}$ is an invertible map, the full map $g\in \Hom_{\Sy}(\Pac, \End_B)$ induces an isomorphism 
$ G : \Pac(B) \to \Pac(B)$
of $\Pac$-coalgebras, with inverse $G^{-1}$ given by the formulae of \cite[Theorem~$10.4.1$]{LodayVallette12}.
Let $\nu \in \Hom_{\Sy}(\Pac, \End_B)$ denote the $\PP_\infty$-algebra structure on $B$, with  corresponding codifferential of $\Pac(B)$ denoted by $d_\nu$. We consider the square-zero degree $-1$ map on $\Pac(B)$ given by 
$d_\xi:=Gd_\nu G^{-1}$.
The following commutative diagram shows that $d_\xi$ is a coderivation. 
$$\xymatrix@C=35pt{\Pac(B) \ar[r]^{G^{-1}}  \ar[d]^{\Delta(B)} \ar@/^5ex/[rrr]^{d_\xi}&  \Pac(B) \ar[r]^{d_\nu} \ar[d]^{\Delta(B)}&  \Pac(B)\ar[r]^{G} \ar[d]^{\Delta(B)} &  \Pac(B)\ar[d]^{\Delta(B)} \\
\Pac \circ\Pac(B)   \ar[r]^{\Pac \circ G^{-1}}   \ar@/_5ex/[rrr]^{\Pac \circ' d_\xi} & \Pac \circ \Pac(B)   \ar[r]^{\Pac \circ' d_\nu}&  \Pac \circ\Pac(B) 
 \ar[r]^{\Pac \circ G} & \Pac \circ \Pac(B)} $$
So it defines another $\PP_\infty$-algebra structure $C:=(B, d_B, \xi)$ on the underlying chain complex $B$, such that the map $g : B \rightsquigarrow C$ becomes an $\infty$-isomorphism. This concludes the proof. 

\item The second point is shown by the same kind of arguments, where one has to use a splitting of the epimorphism 
$g_{(0)} : B \epi C$ this time. 

\end{enumerate}
\end{proof}

\subsection{Proof of Theorem~\ref{theo:MCcoalgPinfty}}

\begin{proof}(of Theorem~\ref{theo:MCcoalgPinfty}--$(1)$)

\begin{itemize}
\item[$(\text{MC}\, 1')$] [\emph{Finite products and pullbacks of fibrations}]  
This is a direct corollary of the fibrant objects category structure \cite{Brown73}. Let us just make the product construction explicit. 
Let $(A, d_A, \mu)$ and $(B, d_B, \nu )$ be two $\PP_\infty$-algebras. Their product is given by $A\oplus B$ with $\PP_\infty$-algebra structure:
$$\PP^{\ac} \xrightarrow{\mu + \nu} {\End}_A \oplus {\End}_B \to {\End}_{A \oplus B}\ . 
$$
The structures maps are the classical projections $A\oplus B \epi A$ and $A\oplus B \epi B$. Any pair 
$C \stackrel{f}{\rightsquigarrow} A$ and $C \stackrel{f}{\rightsquigarrow} B$ of $\infty$-morphisms extend to the following unique $\infty$-morphism: 
$$\PP^{\ac} \xrightarrow{f+g} {\End}^C_A \oplus {\End}^C_B \to {\End}^C_{A \oplus B}\ . 
$$

\item[$(\text{MC}\, 2)\ $] [\emph{Two out of three}] Straightforward.

\item[$(\text{MC}\, 3)\ $] [\emph{Retracts}] Straightforward.

\item[$(\text{MC}\, 4)\ $] [\emph{Lifting property}]  
We consider the following commutative diagram in the category $\infty$\textsf{-}$\PP_\infty$\textsf{-alg} 
$$\xymatrix@M=6pt{A  \ar@{~>}[r]^h  \ar@{>~>}[d]^f &    C\ \ar@{~>>}[d]^g \\
 B \ar@{~>}[r]^k   &   D\ , } $$
where $f$ is a cofibration and where $g$ is a fibration. Using Proposition~\ref{prop:strification}, we can equivalently suppose that the morphisms $f$ and $g$ are strict. Let us prove by induction on the weight $n$ the existence of  a lifting $l : B \rightsquigarrow C$ of the diagram 
\begin{eqnarray}\label{eqn:Diagram}
\xymatrix@M=6pt{A  \ar@{~>}[r]^h  \ar@{>->}[d]^f &    C\ \ar@{->>}[d]^g \\
 B \ar@{~>}[r]^k  \ar@{..>}[ur]^{l}  &   D\ , } 
 \end{eqnarray}
when either $f$ or $g$ is  a quasi-isomorphism. The lifting property $(\text{MC}\, 4)$ of the model category structure on unbounded chain complex \cite{Hovey99} provides us with a chain map $l_{(0)} : (B,d_B) \to (C, d_C)$ such that the following diagram commutes 
$$\xymatrix@M=6pt{(A,d_A)  \ar[r]^{h_{(0)}}  \ar@{>->}[d]^f &    (C,d_C)\ \ar@{->>}[d]^g \\
 (B,d_B) \ar[r]^{k_{(0)}}  \ar@{..>}[ur]^{l_{(0)}}  &   (D,d_D)\ . } $$
Suppose constructed the components $l_{(0)}, l_{(1)}, \ldots, l_{(n-1)}$ of the map $l$ such that Diagram~(\ref{eqn:Diagram}) commutes up to weight $n-1$.  Let us look for  a map $l_{(n)} \in 
\Hom_{\Sy}({\Pac}^{(n)}, \End^B_C)$ 
such that the diagram~(\ref{eqn:Diagram}) commutes up in weight $n$ and such that Equation~(\ref{eqn:inftyMorphEquiv})  is satisfied in weight $n$, i.e.: 

\begin{eqnarray*}
&\text{(a)}& \ f^* l_{(n)}=h_{(n)}\ , \\
&\text{(b)}& \ g_* l_{(n)}=k_{(n)}\ , \\
&\text{(c)}& \ \partial^A_B\, l_{(n)} =  \sum_{k=1}^n l_{(n-k)} \ast  \mu_{(k)} - \sum_{k=1}^n \nu_{(k)} \circledast{l}_{(n-k)}+  l_{(n-1)} d_\varphi  \ ,
\end{eqnarray*}
where $\mu$ and $\nu$ stand respectively for the $\PP_\infty$-algebra structures on $B$ and $C$. We consider a retraction $r : B \epi A$ of $f$, $rf=\id_A$, and a section $s : D \mono C$ of $g$, $gs=\id_D$, in the category of graded modules. The map $\ell$ defined by 
$$\ell:= r^* h_{(n)} +s_*k_{(n)} -(sg)_* r^* h_{(n)} $$
is a solution to Equations~(a) and (b). Let us denote by $\widetilde{l}_{(n)}$ the right-hand side of Equation~(c), as in Appendix~\ref{App:Obstruction}. Since 
$f$ is a morphism  of $\PP_\infty$-algebras, then one can see, by a direct computation from the definition, that the obstruction 
$\widetilde{\big( f^*l \big)}_{(n)}$ to lift $f^*l=l\circledcirc f$ is equal to $f^* \widetilde{l}_{(n)}$. This implies 
$$ f^*\big(\partial^B_C\, \ell -  \widetilde{l}_{(n)}\big)=\partial^A_C \big( f^* \ell   \big)-\widetilde{\big( f^*l \big)}_{(n)}
= \partial^A_C \big( h_{(n)}   \big)-\widetilde{h}_{(n)}=0Ê\ .
$$
In the same way, the relation 
$\widetilde{\big( g_*l \big)}_{(n)}=g_* \widetilde{l}_{(n)}$
 gives 
 $$ g_*\big(\partial^B_C\, \ell -  \widetilde{l}_{(n)}\big)
=\partial^B_D \big( g_* \ell   \big)-\widetilde{\big( g_*l \big)}_{(n)}
=\partial^A_C \big( k_{(n)}   \big)-\widetilde{k}_{(n)}=0Ê\ .
$$
Let $\lambda: {\Pac}^{(n)}(B) \to C$ be the image of $\partial^B_C\, \ell -  \widetilde{l}_{(n)}$ under the isomorphism 
$\Hom_{\Sy}({\Pac}^{(n)}, \End^B_C)\cong \Hom({\Pac}^{(n)}(B), C)$.
Since $\lambda \circ {\Pac}^{(n)}(f)=0$ and since $g \circ \lambda=0$, then the map $\lambda$ factors through a map $\bar \lambda : \coker {\Pac}^{(n)}(f) \to \ker g$, that is 
$\lambda=i \bar{\lambda} p$, where $i$ and $p$ are the respective canonical injection and projection. 
If $f$ is a quasi-isomorphism, then so is the map ${\Pac}^{(n)}(f)$, by the operadic K\"unneth formula \cite[Theorem~$6.2.3$]{LodayVallette12} and hence
the chain complex $\coker {\Pac}^{(n)}(f)$ is acyclic. 
Respectively, 
if $g$ is a quasi-isomorphism, then the chain complex $\ker g$ is acyclic. 
Theorem~\ref{thm:Obstruction} shows that $\lambda$ is a cycle for the differential $(\partial^B_C)_*$, then so is $\bar \lambda$. 
Hence, in either of the two aforementioned cases, the cycle $\bar \lambda$ is boundary element : 
$\bar{\lambda}=\partial(\theta)$. By a slight abuse of notation, we denote by $i\theta p$ the corresponding element in  
$\Hom_{\Sy}({\Pac}^{(n)}, \End^B_C)$. 
Finally, we consider the element 
$$l_{(n)}:= \ell - i \theta  p \ ,$$
which  satisfies Equations~(a), (b) and (c).

\item[$(\text{MC}\, 5)\ $] [\emph{Factorization}] 
 Let $f : A \rightsquigarrow B$ be an $\infty$-morphism of $\PP_\infty$-algebras. 
\begin{itemize}
\item[(a)]

We denote by $D:=\coker( f_{(0)})$ the cokernel of this first component. 
We consider the mapping cone $C:=D \oplus sD$ of the identity $\id_{sD}$ of the suspension of  $D$ as a $\PP_\infty$-algebra with trivial structure.  The product $A\prod C$ of the two $\PP_\infty$-algebras $A$ and $C$ is given by the underlying chain complex $A\oplus C$, see the above proof of the axiom~$(\text{MC}\, 1')$. Let $i : A \to A\prod C$ be the morphism of $\PP_\infty$-algebras defined by $\id_A \oplus 0$ and let $p_{(0)} : A\oplus C \epi B\cong \Im (f_{(0)})\oplus D$ be 
the chain map defined by the sum $f_{(0)}\oplus \mathrm{proj}_D$ of the first component of $f$ and the canonical projection onto $D$. So this provides us with a lifting of the following commutative diagram in the category of chain complexes 
$$\xymatrix@M=6pt@R=30pt@C=30pt{A  \ar[r]^{f_{(0)}}  \ar@{>->}[d]^\sim_{i} &   B\ \ar@{->>}[d] \\
 A\oplus C \ar[r]  \ar@{..>>}[ur]^(0.55){p_{(0)}}  &  0\ . } $$
 Using the same arguments as in the aforementioned proof of Axiom~$(\text{MC}\, 4)$, there exists an $\infty$-morphism $p : A\prod C \rightsquigarrow B$, extending $p_{(0)}$ such that the following diagram commutes 
 $$\xymatrix@M=6pt@R=30pt@C=30pt{A  \ar@{~>}[r]^{f}  \ar@{>->}[d]^\sim_i &   B\ \ar@{->>}[d] \\
 A\prod C \ar[r]  \ar@{~>>}[ur]^(0.52){p}  &  0\ . } $$
So the factorization $f=pi$ concludes the proof. 

\item[(b)] Let $C:=s^{-1}A\oplus A$ be the mapping cone of the identity $\id_A$ of $A$. 
By Corollary~\ref{coro:ExtensionQIinftyMor}, the canonical inclusion $A \mono C$ extends to an $\infty$-monomorphism denoted by 
$j : A \rightsquigarrow C$. We consider the product $i$ of the two $\infty$-morphisms $f$ and $j$: 
$$\xymatrix@M=6pt{ & A \ar@{~>}[ld]_f\ar@{>~>}[dd]_i \ar@{>~>}[rd]^j & \\
B & &   C \\
&  B\prod C\ . \ar@{~>>}[lu]^p_\sim\ar@{~>}[ur]& }$$
Since the underlying chain complex of the product $B\prod C$ is equal to $B\oplus C$, 
the $\infty$-morphism $i$ is an cofibration and the projection $p$ is an acyclic fibration. So the factorization $f=pi$ concludes the proof. 

\end{itemize}

\end{itemize}

\end{proof}

\subsection{Relationship with the model category structure on conilpotent dg $\Pac$-coalgebras}
Since the bar construction $\widetilde{\B}_\iota$ provides us with an isomorphism of categories 
$$\widetilde{\B}_\iota\  : \ \infty\textsf{-}\PP_\infty\textsf{-alg} \cong  \textsf{quasi}\textsf{-}\textsf{free}\  \PP^{\ac}\textsf{-}\textsf{coalg}\ , $$
we can compare  the model category structure without equalizers on $\PP_\infty$-algebras (Theorem~\ref{theo:MCcoalgPinfty}) with the model category structure on conilpotent $\PP^{\ac}$-coalgebras (Theorem~\ref{theo:MCcoalg}). 
The following proposition shows that the various notions of weak equivalences, cofibrations and fibrations agree. 

\begin{prop}\label{prop:InftyqiWE}\leavevmode

\begin{enumerate}

\item An $\infty$-morphism $f$ is a weak-equivalence of $\PP_\infty$-algebras 
if and only if its image under the bar construction $\widetilde{\B}_\iota \, f $ is a weak equivalence of 
conilpotent dg $\Pac$-coalgebras. 

\item An $\infty$-morphism $f$ is a cofibration of $\PP_\infty$-algebras 
if and only if its image under the bar construction $\widetilde{\B}_\iota \, f$ is a cofibration of 
conilpotent dg $\Pac$-coalgebras. 

\item An $\infty$-morphism $f$ is a fibration of $\PP_\infty$-algebras 
if and only if its image under the bar construction $\widetilde{\B}_\iota \, f $ is a fibration of 
conilpotent dg $\Pac$-coalgebras. 

\end{enumerate}
\end{prop} 

\begin{proof}$ \ $ 

\begin{enumerate}
\item This is Proposition~$11.4.7$ of \cite{LodayVallette12}.

\item Let  $f : A \rightsquigarrow A'$ be an $\infty$-morphism
of $\PP_\infty$-algebras. If  $\widetilde{\B}_\iota \, f : \widetilde{\B}_\iota A \mono \widetilde{\B}_\iota  A'$ is a cofibration of conilpotent dg $\Pac$-coalgebras, then it is a monomorphism by definition. So its restriction to $A$ is again a monomorphism. Since this restriction is equal to the composite 
$$A \xrightarrow{f_{(0)}} A' \mono \Pac(A')\ ,  $$
this implies that the first component $f_{(0)}$ is a monomorphism. 

\noindent 
In the other way round, suppose that the $\infty$-morphism $f : A \rightsquigarrow A'$ is an $\infty$-monomorphism, i.e. $f_{(0)} : A \mono A'$ is injective. Let $r_{(0)} : A' \epi A$ be a retraction of $f_{(0)}$. The formula of \cite[Theorem~$10.4.1$]{LodayVallette12} produces an $\infty$-morphism $r$, which is right inverse to $f$. Therefore, the image $\widetilde{\B}_\iota \, r$ is a right inverse to $\widetilde{\B}_\iota \, f$, which proves that this latter one is a monomorphism. 

\item Let us first recall that axioms $(\text{MC}\, 3)$ and $(\text{MC}\, 5)$ imply that fibrations are characterized by the right lifting property with respect to acyclic cofibrations. So this characterization also holds for $\infty$-epimorphisms in the model category without equalizers of Theorem~\ref{theo:MCcoalgPinfty}. 

\smallskip

\noindent
 Let  $f : A \rightsquigarrow A'$ be an $\infty$-morphism of $\PP_\infty$-algebras. Suppose that   $\widetilde{\B}_\iota \, f : \widetilde{\B}_\iota A \epi \widetilde{\B}_\iota  A'$ is a fibration of conilpotent dg $\Pac$-coalgebras. We consider a commutative diagram in $\infty\textsf{-}\PP_\infty\textsf{-alg}$
\begin{eqnarray}\label{eqn:Diagrambis}
\xymatrix@M=6pt{B  \ar@{~>}[r]^h  \ar@{>~>}[d]^\sim_g &    A\ \ar@{~>>}[d]^f \\
 B' \ar@{~>}[r]^k   &   A'\ , } 
 \end{eqnarray}
where $g$ is an  acyclic cofibration. Its image 
$\widetilde{\B}_\iota \, g : \widetilde{\B}_\iota B \to \widetilde{\B}_\iota  B'$
under the bar construction functor is an acyclic cofibration of conilpotent dg $\Pac$-coalgebras by the two previous points $(1)$ and $(2)$. So, by the axiom $(\text{MC}\, 4)$ of Theorem~\ref{theo:MCcoalg},  there exists a lifting map 
${\widetilde{\B}_\iota \, l}$ in following  diagram
$$\xymatrix@R=30pt@C=30pt@M=5pt{\widetilde{\B}_\iota B \ar[r]  \ar@{>->}[d]_{\widetilde{\B}_\iota \, g}^\sim & \widetilde{\B}_\iota A\  \ar@{->>}[d]^{\widetilde{\B}_\iota \, f} \\
\widetilde{\B}_\iota B' \ar[r] \ar@{..>}[ur]^{{\widetilde{\B}_\iota \, l}}  & \widetilde{\B}_\iota A'\ , }$$
which proves that $l$ is a lifting map in Diagram~\eqref{eqn:Diagrambis}. So the $\infty$-morphism 
$f : A \rightsquigarrow A'$
is a fibration. 

\noindent
In the other way round, suppose that the $\infty$-morphism $f : A \rightsquigarrow A'$ is a fibration of $\PP_\infty$-algebras and let  
$$\xymatrix@R=30pt@C=30pt@M=5pt{C \ar[r]^H  \ar@{>->}[d]_{G}^\sim & \widetilde{\B}_\iota A\  \ar@{->}[d]^{\widetilde{\B}_\iota \, f} \\
C' \ar[r]^K   & \widetilde{\B}_\iota A'\  }$$
be a commutative diagram in the category of conilpotent dg $\Pac$-coalgebras, where $G$ is an acyclic cofibration. 
We consider the diagram 
$$\xymatrix@R=30pt@C=30pt@M=5pt{
C \ar[rr]^H \ar@{>->}[dr]^{\upsilon_\kappa C}_\sim  \ar@{>->}[dd]_{G}^\sim & &  \widetilde{\B}_\iota A\  \ar@{->}[dd]^{\widetilde{\B}_\iota \, f} \\
&\B_\kappa \Omega_\kappa C\  \ar@{>->}[dd]_(0.7)\sim^(0.7){\B_\kappa \Omega_\kappa G}&\\
C' \ar[rr]^(0.56)K | (0.48)\hole \ar[dr] &  & \widetilde{\B}_\iota A'\ \\
&\B_\kappa \Omega_\kappa C' \ .& }$$
Since the unit $\upsilon_\kappa C$ of the bar-cobar adjunction is an acyclic cofibration by Theorem~\ref{Thm:KoszulBarCobarRes}--$(2)$ and since $\widetilde{\B}_\iota A$ is a fibrant dg $\Pac$-coalgebra, then 
there exists a morphism $H' : \B_\kappa \Omega_\kappa C \to \widetilde{\B}_\iota A$ which factors $H$, i.e. $H=H' \upsilon_\kappa C$. Since $G$ is a weak-equivalence, then $\Omega_\kappa G$ is a quasi-isomorphism, by definition, and so 
$ \B_\kappa \Omega_\kappa G$ is a weak-equivalence by point~$(1)$. 
In the same way, since $G$ is a cofibration, then it is a monomorphism, by definition, and so 
is $\Omega_\kappa G=\id_\PP(G) : \PP(C) \mono \PP(C')$; point~$(2)$ shows that $\B_\kappa \Omega_\kappa G$
is  a cofibration. All together, this proves that the map $\B_\kappa \Omega_\kappa G$ is an acyclic cofibration of conilpotent dg $\Pac$-coalgebras, and, since  $\widetilde{\B}_\iota A'$ is fibrant, then 
there exists a morphism $K' : \B_\kappa \Omega_\kappa C' \to \widetilde{\B}_\iota A'$ which factors 
$(\widetilde{\B}_\iota \, f) H'$, i.e. $(\widetilde{\B}_\iota \, f) H' = K' (\B_\kappa \Omega_\kappa G)   $.
Finally, we apply the lifting property axiom $(\text{MC}\, 4)$ of Theorem~\ref{theo:MCcoalgPinfty} in the diagram 
$$\xymatrix@R=30pt@C=30pt@M=5pt{
\Omega_\kappa  C  \ar@{~>}[r]^{h'}  \ar@{>->}[d]_{\Omega_\kappa G}^\sim &  A\  
\ar@{~>>}[d]^{f} \\
\Omega_\kappa  C'  \ar@{~>}[r]^{k'} \ar@{..>}[ur]  &  A' }$$
to conclude that $\widetilde{\B}_\iota \, f$ is a fibration.
\end{enumerate}
\end{proof}


\appendix 

\section{Obstruction theory for infinity-morphisms}\label{App:Obstruction}
In this appendix, we settle the obstruction theory for $\infty$-morphisms of homotopy $\PP$-algebras. \\

Recall from Theorem~\ref{theo:EquivInftyMorph}, that an $\infty$-morphism $f : A \rightsquigarrow B$ between two  $\PP_\infty$-algebras $(A, d_A, \mu)$ and $(B, d_B, \nu )$ is a map 
$f : {\PP}^{\ac} \to \End^A_B$
satisfying Equation~\eqref{eqn:inftyMorphEquiv}
\begin{eqnarray*}
\partial(f) =f * \mu - \nu \circledast f 
\end{eqnarray*}
 in $\Hom_\Sy(\PP^{\ac},\End_B^A)$.

For any $n\ge 0$, we denote by 
$\mu_{(n)}$, $\nu_{(n)}$, and $f_{(n)}$ the respective restrictions of the maps $\mu$, $\nu$, and $f$ to 
the weight $n$ part ${\PP^{\ac}}^{(n)}$ of the cooperad $\PP^{\ac}$. Using these notations, Equation~\eqref{eqn:inftyMorphEquiv} becomes 
\begin{eqnarray}\label{eqn:InftyMor(n)}
\partial^A_B f_{(n)} - f_{(n-1)} d_\varphi =  \sum_{k=1}^n f_{(n-k)} \ast  \mu_{(k)} - \sum_{k=1}^n \nu_{(k)} \circledast{f}_{(n-k)} 
\end{eqnarray}
on ${\PP^{\ac}}^{(n)}$, for any $n\ge 0$, where  $\nu_{(k)} \circledast{f}_{(n-k)} $ means, by a slight abuse of notations, that  the total weight of the various maps $f$ involved on the right-hand side is equal to $n-k$. 

\begin{theo}\label{thm:Obstruction}
Let $(A, d_A, \mu)$ and $(B, d_B, \nu )$ be two $\PP_\infty$-algebras. 
Let $n\ge 0$ and suppose given $f_{(0)}, \allowbreak \ldots,\allowbreak  f_{(n-1)}$ satisfying Equation~\eqref{eqn:InftyMor(n)} up to $n-1$. 
The element 
$$  \widetilde{f}_{(n)}:= \sum_{k=1}^n f_{(n-k)} \ast  \mu_{(k)} - \sum_{k=1}^n \nu_{(k)} \circledast {f}_{(n-k)} + f_{(n-1)} d_\varphi$$ 
is a cycle in the chain complex $\big(\Hom_\Sy(\PP^{\ac},\End_B^A), (\partial^A_B)_* \big)$. 
Therefore, there exists an element $f_{(n)}$ satisfying Equation~\eqref{eqn:inftyMorphEquiv} at weight $n$ if and only if 
the cycle $\widetilde{f}_{(n)}$ is a boundary element. 
\end{theo}

\begin{proof}
Let us prove that $\partial^A_B \widetilde{f}_{(n)}=0$; the second statement is then straightforward. We have $\partial^A_B \widetilde{f}_{(n)}=$
\begin{eqnarray*}
&&\sum_{k=1}^n \big( (\pAB f_{(n-k)}) \ast  \mu_{(k)} + f_{(n-k)} \ast  (\partial_A \mu_{(k)})
-  (\partial_B \nu_{(k)}) \circledast{f}_{(n-k)} 
+  \nu_{(k)} \circledast {(f; \pAB f)}_{(n-k)} \big)\\
&&+ \pAB f_{(n-1)} d_\varphi \ , 
\end{eqnarray*}
where the notation $(f ; \pAB f)$ means that we apply the definition of the product $\circledast$ with many $f$ but one $\pAB f$. (With the notations of \cite{LodayVallette12}, this coincides to composing by $\nu\circ (f; \pAB f)$.)  
Applying Equation~\eqref{eqn:InftyMor(n)} at weight strictly less than $n$ and the Maurer--Cartan equation 
$$\partial_A \mu_{(n)} = -\sum_{k=1}^n \mu_{(k)}\star \mu_{(n-k)} - \mu_{(n-1)}d_\varphi $$
for $\mu$ and $\nu$ respectively, we get 
\begin{eqnarray*}
\partial^A_B \widetilde{f}_{(n)}&=&
\sum_{k+l+m=n} \big( 
(f_{(k)} \ast \mu_{(l)})\ast \mu_{(m)} 
- f_{(k)} \ast (\mu_{(l)}\star \mu_{(m)})
+ (\nu_{(k)}\star \nu_{(l)})\circledast  {f}_{(m)}\\
&&\qquad\qquad
- \nu_{(k)}\circledast  {(f; \nu\circledast f)}_{(l+m)}
+ \nu_{(k)}\circledast  {(f; f\ast \mu)}_{(l+m)} 
- (\nu_{(k)}\circledast  {f}_{(l)})\ast \mu_{(m)}
\big) \\
&&+\sum_{k=1}^{n-1} \big(
(f_{(n-k-1)} d_\varphi)\ast \mu_{(k)}
- f_{(n-k-1)}\ast (\mu_{(k)} d_\varphi)
+ (f_{(n-k-1)} \ast \mu_{(k)})d_\varphi \\
&&\qquad\quad
+ (\nu_{(k)}d_\varphi) \circledast  {f}_{n-k-1}) 
- (\nu_{(k)} \circledast  {f}_{n-k-1}) d_\varphi
+ \nu_{(k)} \circledast  {(f; f_{(\bullet-1)}d_\varphi)}_{(n-k-1)}
\big)\\
&&
+ f_{(n-2)} (d_\varphi)^2 \ .
\end{eqnarray*}
Since $\mu$ has degree $-1$, the preLie relations of the operations $\star$ and $\ast$ imply 
$(f\ast\mu)\ast \mu = f\ast (\mu \star \mu)$.
The coassociativity of the decomposition coproduct $\Delta$ of the cooperad $\Pac$ implies 
$(\nu \star \nu)\circledast f = \nu (f; \nu \circledast f)$ and 
$(\nu \circledast f)\ast \mu  = \nu \circledast (f;f \circledast \mu)$.
Since $d_\varphi$ is a coderivation of the cooperad $\Pac$, it implies 
$(f\ast \mu)d_\varphi = f\ast (\mu d_\varphi)-(f d_\varphi)\ast \mu$ and 
$(\nu\circledast f)d_\varphi = (\nu d_\varphi)\circledast f+ \nu \circledast (f; f_{(\bullet-1)}d_\varphi)$. 
Finally, the coderivation $d_\varphi$ squares to zero, which concludes the proof. 
\end{proof}

\begin{coro}\label{coro:ExtensionQIinftyMor}
Let $(A, d_A, \mu)$ be a $\PP_\infty$-algebra and let $(B, d_B)$ be an acyclic chain complex, viewed as a trivial $\PP_\infty$-algebra. 
Any chain map $(A, d_A)\to (B,d_B)$ extends to an $\infty$-morphism 
$(A, d_A, \mu)\rightsquigarrow (B,d_B,0)$.
\end{coro}

\begin{proof}
We prove the existence of a series of maps $f_{(n)}$, for $n\ge 0$ satisfying Equation~\eqref{eqn:InftyMor(n)} by induction on $n$ using Theorem~\ref{thm:Obstruction}. 
Let us denote the  map $A \to B$ by $f_{(0)}$. Since this is a chain map, it satisfies Equation~\eqref{eqn:InftyMor(n)} for $n=0$. 
Since the chain complex $(B, d_B)$ is acyclic, then so is the chain complex $\big(\Hom_\Sy(\PP^{\ac},\End_B^A), (\partial^A_B)_* \big)$. Therefore, all the obstructions vanish and Theorem~\ref{thm:Obstruction} applies.
\end{proof}

\section{A technical lemma}\label{app:TechLemma}

\begin{lemm}\label{lemm:AcyclicCofibration}
Let $A$ be a dg $\Po$-algebra and let $D$ be a conilpotent dg $\Pac$-coalgebra. Let $p : A \epi \Omega_\kappa D$ be a fibration of dg $\Po$-algebras. The morphism $j : \B_\kappa A \times_{\B_\kappa \Omega_\kappa D} D \stackrel{\sim}{\mono} \B_\kappa A$, produced by the pullback diagram 
$$\xymatrix@M=5pt@R=30pt@C=30pt{\B_\kappa A \times_{\B_\kappa \Omega_\kappa D} D \ar[r] \ar@{>->}[d]_j^\sim  
\ar@{}[rd] | (0.35)\pullback  & D  \ar[d]^{\upsilon_\kappa D}   \\
\B_\kappa A  \ar[r]^{\B_\kappa p} &  \B_\kappa \Omega_\kappa D \ ,} $$ 
is an acyclic cofibration of conilpotent dg $\Pac$-coalgebras.
\end{lemm}

\begin{proof}
We consider the kernel $K:=\Ker(p)$ of the map $p : A \epi \Omega_\kappa D$, which is a sub-dg $\PP$-algebra of $A$. The short exact sequence 
$$\xymatrix@C=20pt{0 \ar[r] & K \ar@{>->}[r] & A \ar@{->>}[r] &  \ar@/_1pc/@{..>}[l] \Ok D =\PP(D) \ar[r] & 0} $$
of dg $\PP$-algebras splits in the category of graded $\PP$-algebras since the underlying $\PP$-algebra of $\Ok D=\PP(D)$ is free. The induced isomorphism of graded $\PP$-algebras $A \cong K \oplus \PP(D)$ becomes an isomorphism of dg $\PP$-algebras when the right-hand side is equipped with the transferred differential, which the sum of the following three terms 
$$d_K \ : \ K \to K, \quad d_{\Ok D}\ : \ \PP(D) \to  \PP(D), \ \text{and} \quad  d'\ : \ \PP(D) \to K\ . $$
Notice that $K \oplus \PP(D)$ endowed with the $\PP$-algebra structure given by 
$$\PP( K \oplus \PP(D)) \epi \PP( K ) \oplus \PP(  \PP(D)) \xrightarrow{\gamma_K\oplus \gamma_{\PP(D)}} K \oplus \PP(D)$$
 is the product of $K$ and $\PP(D)$ in the category of graded $\PP$-algebras. 
Since the bar construction is right adjoint, it preserves the limits and thus the products. This induces the following two isomorphisms of conilpotent dg  $\Pac$-coalgebras 
$$\Bk A  \cong  \Bk K \times \Bk \Ok D, \quad \text{and} \quad  \Bk A \times_{\Bk \Ok D} D  \cong  \Bk K \times D \ ,  $$
with both right-hand sides equipped with an additional differential coming from $d'$.
Under these identifications, the initial pullback becomes 
$$\xymatrix@M=5pt@R=30pt@C=30pt{ \Bk K \times D  \ar[r]^(0.58){\text{proj}} \ar[d]_{\id \times \upsilon_\kappa D}  
\ar@{}[rd] | (0.35)\pullback  & D  \ar[d]^{\upsilon_\kappa D}   \\
 \Bk K \times \Bk \Ok D  \ar[r]^(0.58){\text{proj}} &  \B_\kappa \Omega_\kappa D \ .} $$ 
 Since the unit of adjunction $\upsilon_\kappa D$ is monomorphism, then so is the map $j$, which is therefore a cofibration. \\

It remains to prove that $\id \times \upsilon_\kappa D$ is again a weak-equivalence when the twisted differentials, coming from $d'$, are taken into account. We will prove that this is a filtered quasi-isomorphism with the same kind of filtrations as in the proof of Theorem~\ref{Thm:KoszulBarCobarRes}. 

We notice first that the product of two $\Pac$-coalgebras is given by an equalizer dual to the coequalizer introduced at the beginning of Section~\ref{subsec:Fib-Cofib} to describe coproducts of $\PP$-algebras. So the product $\Pac(K)\times D$ is a conilpotent sub-$\Pac$-coalgebra of the conilpotent cofree $\Pac$-coalgebra $\Pac(K \oplus D)$. We filter the latter one by 
$$\F_n \, \big( \Pac(K \oplus D)\big) := \sum_{k\ge 1,  \atop n_1+\cdots+n_k \leq n} \Pac(k)\t_{\Sy_k} \big((K\oplus F_{n_1} D)\t  \cdots \t (K\oplus  F_{n_k} D)\big) $$
and we denote by $\F_n \, (\Bk K \times D)$ the induced filtration 
on the product $\Bk K \times D$. In the same way, we filter the product $\Bk K \times \Bk \Ok D$, whose underlying conilpotent $\Pac$-coalgebra is 
$\Pac(K\oplus \PP(D))$,
 by 
$$\F_n \,( \Bk K \times \Bk \Ok D) := \sum_{k\ge 1,  \atop n_1+\cdots+n_k \leq n} \Pac(k)\t_{\Sy_k} \big((K\oplus F_{n_1} \PP(D))\t  \cdots \t (K\oplus  F_{n_k} \PP(D))\big),  $$
where 
$$F_{n} \PP(D) :=\sum_{k\ge 1,  \atop n_1+\cdots+n_k \leq n} \PP(k)\t_{\Sy_k} (F_{n_1} D\t  \cdots \t   F_{n_k} D)\ .$$
The respective differentials and the map $\id \times \upsilon_\kappa D$ preserve these two filtrations. Let us now prove that the associated map 
$\gr (\id \times \upsilon_\kappa D)=\id \times \gr (\upsilon_\kappa D)$ is a quasi-isomorphism with the twisted differentials coming for $d'$.
We now consider the filtration $\Fil_n \, (\Bk K \times \gr D)$ induced by 
$$\Fil_n \, \big( \Pac(K \oplus \gr D)\big) := \sum_{k\ge 1,  \atop n_1+\cdots+n_k+k \leq n} \Pac(k)\t_{\Sy_k} 
\big((K\oplus \gr_{n_1} D)\t  \cdots \t (K\oplus  \gr_{n_k} D)\big) \ .$$
We introduce the filtration 
$\Fil_n  (\Bk K \times \Bk \Ok\, \gr D)$ given by 
$$\Fil_n \,\big( \Pac(K\oplus \PP(\gr D))\big) := \sum_{k\ge 1,  \atop n_1+\cdots+n_k \leq n} \Pac(k)\t_{\Sy_k} \big((K\oplus \Fil_{n_1} \PP(\gr D))\t  \cdots \t (K\oplus  \Fil_{n_k} \PP(\gr D))\big),  $$
where 
$$\Fil_{n} \PP(\gr D) :=\sum_{k\ge 1,  \atop n_1+\cdots+n_k +k \leq n} \PP(k)\t_{\Sy_k} (\gr_{n_1} D\t  \cdots \t   \gr_{n_k} D)\ .$$
The respective differentials and the map $\id \times \gr (\upsilon_\kappa D)$ preserve these two  filtrations. By their definitions, the part of the differential coming from $d'$, and only this part, is killed on the first page of the respective spectral sequences. By the same arguments as in the proof of Theorem~\ref{Thm:KoszulBarCobarRes}, we get a quasi-isomorphism between these first pages. Since the two aforementioned filtrations are bounded below and exhaustive, we conclude by the classical  convergence theorem of spectral sequences \cite[Chapter~$11$]{MacLane95}.
\end{proof}

\begin{center}
\textsc{Acknowledgements}
\end{center}
It is a pleasure to thank Joana Cirici, Volodya Dotsenko, Gabriel C. Drummond-Cole, Joey Hirsh,
Bernhard Keller, Brice Le Grignou,
 Emily Riehl, and David White for useful discussions. I would like to express my appreciation to the Simons Center for Geometry and Physics in Stony Brook for  the invitation and the excellent working conditions.  

\bibliographystyle{amsalpha}
\bibliography{bib}

\end{document}